\documentclass[12pt,a4paper]{article}

\usepackage[dvips]{graphicx}
\usepackage{amssymb,latexsym}
\usepackage{mathrsfs}
\usepackage{psfrag}
\usepackage[all]{xy}
\usepackage{color}
\usepackage{verbatim}

\newtheorem{prop}{Proposition}[section]
\newtheorem{theo}[prop]{Theorem}

\newtheorem{defn}[prop]{Definition}
\newtheorem{rem}[prop]{Remark}

\newtheorem{lem}[prop]{Lemma}

 \newcommand{\bx}{\mbox{\boldmath $x$}}

\newcommand{\bv}{\mbox{\boldmath $v$}}

\newcommand{\bu}{\mbox{\boldmath $u$}}

\newcommand{\bN}{\mbox{\boldmath $N$}}

\newcommand{\R}{\mbox{\boldmath $\mathbb{R}$}}

\newenvironment{proof}
{\begin{trivlist} \item[\hskip \labelsep {\bf Proof}\hspace*{3 mm}]}
	{\hfill$\Box$\end{trivlist}}
\newenvironment{acknow}
{\begin{trivlist} \item[\hskip \labelsep {\bf Acknowledgments.}]}
	{\end{trivlist}}

\date{}

\begin{document}

\title{On the multiplicity of umbilic points}
\author{Marco Ant\^onio do Couto Fernandes and Farid Tari}

\maketitle
\begin{abstract}
We introduce an invariant of umbilic points on surfaces in the Euclidean or Minkowski 3-space that
counts the maximum number of stable umbilic points they can split up 
under deformations of the surfaces. 
We call that number the {\it multiplicity of the umbilic point} and establish its properties.
\end{abstract}

\renewcommand{\thefootnote}{\fnsymbol{footnote}}

\section{Introduction}\label{sec:intro}
Umbilic points on surfaces $M$ in the Euclidean space $\mathbb R^3$ or in the Minkowski space $\mathbb R^3_1$ attract a lot of attention (see for example \cite{Gut_Soto_Resenha} for a survey article and \cite{Ando_etal,bruce84, bruce-fidal, Ghomi_Howard,Hasegawa_Tari,porteousBook, Smith_Xavier,SotoLuis, sotogut, CaratheodoryR31} for some recent developments).
The Carath\'eodory conjecture which states that any sufficiently smooth, convex and closed surface in $\mathbb R^3$ has at least two umbilic points is still open. 
(The conjecture is true for surfaces in $\mathbb R^3_1$,  \cite{CaratheodoryR31}.) 

Umbilic points on  a surface $M$ in $\mathbb R^3$ or $\mathbb R^3_1$ 
are the singular points of the lines of principal curvature (these can be extended at points where the induced pseudo metric on $M\subset \mathbb R^3_1$ is degenerate, \cite{IzumiyaTari}). They are the points where all the coefficients of the binary differential equation (BDE) of these lines vanish.

Umbilic points on a generic surface are stable, that is, they persist under small deformations of the surface. However, the BDE of these lines is not stable when deformed within the set of all BDEs; see Figure \ref{fig:bifStar}.

\begin{figure}[htp]
	\begin{center}
		\includegraphics[scale=0.6]{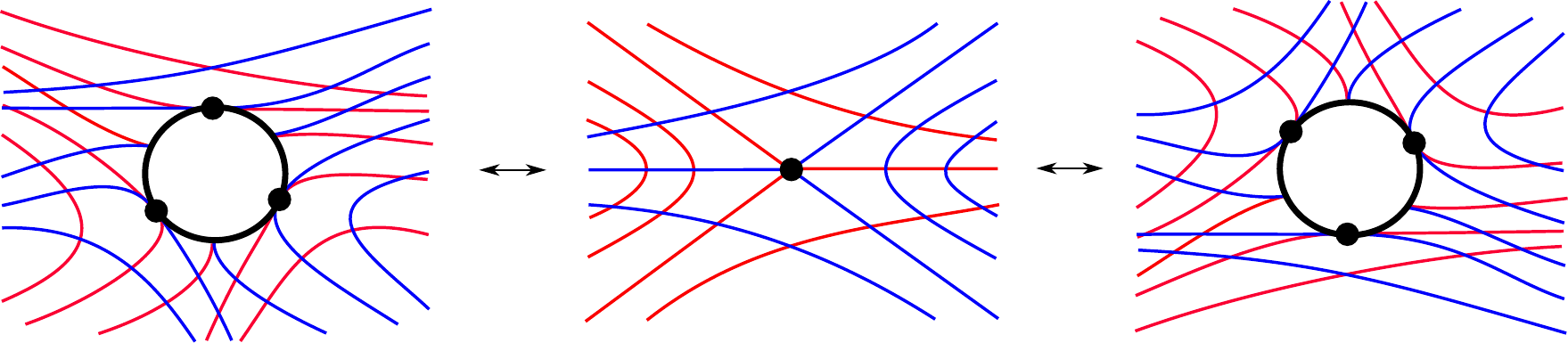}
		\caption{Generic bifurcations, within the set of all BDEs, of the configuration of the solutions of the BDE of the lines of principal curvature at a star umbilic. See \cite{BFT, Fletcher} for a geometric family of BDEs on surfaces realising this bifurcation.}
		\label{fig:bifStar}
	\end{center}
\end{figure}

We introduce in this paper an invariant of a surface $M$ at an umbilic point and call it the {\it multiplicity of the umbilic point}. The multiplicity counts the maximum number of stable umbilic points that can appear under small deformations of the surface.

An invariant of a BDE at its singular points, called the {\it multiplicity of the BDE}, is introduced in \cite{mult} and counts the maximum number of well folded singularities that can appear in a deformation the BDE (in Figure  \ref{fig:bifStar}, the multiplicity is $3$ and the folded singularities are indicated in full discs). 

We establish a relation between the multiplicity of an umbilic point and the multiplicity of the BDE of the lines of curvature at that point (Theorem \ref{theo:mult&multUmb}). We characterise in various ways umbilic points of multiplicity $1$ (Theorems \ref{theo:mult1_Trans_Jetspave}, \ref{theo:mult_1_lpl}, \ref{teo:mult_versal}) and prove that there exist a local deformation of the surface such that the umbilic points surging from a given degenerate umbilic all have multiplicity $1$ (Theorem \ref{theo:defm}). 
We characterise the multiplicity of a degenerate umbilic in terms of the singularities of some special curves on the surface (Theorems \ref{theo:LDMilnor} and \ref{theo:LPLSimple}). 
We apply our results in \S\ref{sec:appl} to codimension 1 and 2 umbilic points studied in \cite{GarciaSoto,GarciaSotoGut} 
and  to surfaces with a cross-cap singularity.

We prove some results in the Appendix (\S \ref{sec:appendix}) on generic configurations of the lines of principal curvature. We write the surface locally as the graph of a function $f$. The configurations depend only the cubic part of $f$ at the umbilic. We stratify the set of these cubics where the open strata consists of stable topologically equivalent configurations of the lines of principal curvature. The results in the Appendix are used in \S \ref{sec:mult}.

We state our results for umbilc points on surfaces in the Minkowski 3-space as  umbilic points on surfaces in the Euclidean 3-space  have the same properties of
spacelike umbilics.

\section{Preliminaries} \label{sec:prel}

The {\it Minkowski space} $(\mathbb{R}_1^{3},\langle ,\rangle )$ is the vector
space $\mathbb{R}^{3}$ endowed with the metric induced by the pseudo-scalar product
$
\langle \bu, \bv\rangle =u_0v_0+u_1v_1-u_2v_2$, for any vectors $\bu=(u_0,u_1,u_2)$ and $\bv =(v_0,v_1, v_2)$ in
$\mathbb{R}^{3}$ (see for example \cite{oNeill} and \cite{Couto_Lymb} for a treatment of the geometry of surfaces in $\mathbb R^3_1$).
A non-zero vector $\bu\in \mathbb R^3_1$ is said to be {\it spacelike} if
$\langle \bu, \bu\rangle>0$, {\it lightlike} if
$\langle \bu, \bu\rangle=0$ and {\it timelike} if $\langle \bu, \bu\rangle<0$.
The norm of a vector $\bu\in \mathbb{R}_1^{3}$ is defined by
$\Vert \bu \Vert=\sqrt{ |\langle \bu, \bu\rangle|}.$

Let $M$ be a smooth and regular surface in $\mathbb R^3_1$ and let 
$\bx:U\subset \mathbb R^2\to M\subset \mathbb{R}_1^{3}$ be a local parametrisation of $M$.
As our definition of the multiplicity is local in nature, we shall
simplify notation  and write $\bx(U)=M$. Let
$$
E=\langle {\bx}_u,{\bx}_u \rangle,\quad
F=\langle{\bx}_u,{\bx}_v\rangle,\quad
G=\langle {\bx}_v,{\bx}_v\rangle
$$
denote the coefficients of the first fundamental form of $M$ with respect to the parametrisation $\bx$, where subscripts denote partial derivatives. The induced (pseudo) metric on $M$ is Lorentzian (resp. Riemannian, degenerate) at ${\rm p}=\bx(u,v)$ if, and only if, $(F^2-EG)(u,v)>0$ 
(resp. $<0$, $=0$). The locus of points on the surface where the metric is degenerate 
is called the {\it locus of degeneracy} and is denoted by $LD$.  
We identify the $LD$ on $M$ with its pre-image in $U$ by $\bx$. Then the $LD$ (in $U$) is given by
$$
LD=\{(u,v)\in U\, |\, (F^2-EG)(u,v)=0\}.
$$

A direction  $(du,dv)\in T_{\rm p}M$ is lightlike, if and only if
\begin{equation}\label{eq:LightBDE}
	Edu^2+2Fdudv+Gdv^2=0.
\end{equation}
Equation (\ref{eq:LightBDE}) has two  (resp. no) solutions when ${{\rm p}}$ is in the Lorentzian (resp. Riemannian) region of $M$.
At points on the $LD$, there is a unique (double) solution of the equation.

At ${{\rm p}} \in M\setminus LD$, we have a well defined unit normal vector (the Gauss map) $\bN=\bx_u\times\bx_v/||\bx_u\times\bx_v||$,
which is timelike (resp. spacelike) if ${\rm p}$ is in the Riemannian (resp. Lorentzian) region of $M$.
(See \cite{pei} for a definition of an $\mathbb RP^2$-valued Gauss map.) The map $A_{{\rm p}}=-d\bN_{{\rm p}}:T_{{\rm p}}M\to T_{{\rm p}}M$ is a self-adjoint
operator on $M\setminus LD$. We denote by
$$\begin{array}{ccccc}
	l&=&-\langle \bN_u,\bx_{u}\rangle&=& \langle \bN,\bx_{uu}\rangle,\\
	m&=&-\langle \bN_u,\bx_{v}\rangle&=&\langle \bN,\bx_{uv}\rangle,\\
	n&=&-\langle \bN_v,\bx_{v}\rangle&=&\langle \bN,\bx_{vv}\rangle
\end{array}
$$
the coefficients of the second fundamental form on $M\setminus LD$.
When $A_{{\rm p}}$ has real eigenvalues $\kappa_1$ and $\kappa_2$, we call them the {\it principal curvatures} and their associated eigenvectors the {\it principal directions}
of $M$ at ${{\rm p}}$. There are always two principal curvatures at each point on the Riemannian part of $M$ but this is not always
true on its Lorentzian part.
A point ${\rm p}$ on $M$ is called an {\it umbilic point} if  $\kappa_1=\kappa_2$ at ${\rm p}$ (i.e., if $A_p$ is a multiple of the identity map). It is called a {\it spacelike umbilic point} (resp. {\it timelike umbilic point}) if ${\rm p}$ is in the Riemannian (resp. Lorentzian) 
part of $M$.

The lines of principal curvature, which are the integral curves of the principal directions, are the solutions of the binary differential equation (BDE)
\begin{equation}\label{eq:principalBDE}
	(Fn-Gm)dv^2+(En-Gl)dvdu+(Em-Fl)du^2=0.
\end{equation}

The discriminant of the BDE (\ref{eq:principalBDE}) 
$$\{
(u,v)\in U | (En-Gl)^2-4(Fn-Gm)(Em-Fl)=0\}
$$
consists of the umbilic points in the Riemannian region of $M$ and is the locus of points in the Lorentzian region where two principal directions coincide and become lightlike. It is labelled {\it Lightlike Principal Locus} $(LPL)$ in \cite{Izu-Machan-Farid, pairs}. The principal directions are orthogonal when there are two of them at a given point; in particular one is spacelike and the other is timelike if the point is in the Lorentzian region.

One can extend the lines of principal curvature across the $LD$ as follows (\cite{IzumiyaTari}).
As equation (\ref{eq:principalBDE}) is homogeneous in $l,m,n$,
we can multiply these coefficients by $||\bx_u\times\bx_v||$ and substitute them in the equation by
\[
\bar{l}=\langle \bx_u\times\bx_v,\bx_{uu}\rangle, \quad
\bar{m}=\langle \bx_u\times\bx_v,\bx_{uv}\rangle,\quad
\bar{n}=\langle \bx_u\times\bx_v,\bx_{vv}\rangle.
\]

The new equation 
\begin{equation}\label{eq:principalLD}
	(\omega_P):	(G\bar{m}-F\bar{n})dv^2+(G\bar{l}-E\bar{n})dudv+(F\bar l-E\bar{m})du^2=0
\end{equation}
is defined at points on the $LD$ and
its solutions are the same as those of equation (\ref{eq:principalBDE}) in the Riemannian and Lorentzian regions of $M$.

We define a {\it lightlike umbilic point} as a point on the $LD$ where all the coefficients of equation  (\ref{eq:principalLD}) vanish (see \cite{CaratheodoryR31}). 
This occurs if, and only if, the $LD$ is singular (\cite{CaratheodoryR31}). The $LPL$ (extended to the $LD$ as the discriminant of equation (\ref{eq:principalLD})) 
passes through lightlike umbilic points and is also singular at that point \cite{MarcoThesis}.

We use here some concepts of singularity theory (see \cite{arnold} and \cite{wall}). Two germs of functions $f_i:\mathbb R^n,0\to \mathbb R,0, i=1,2$ are said to be $\mathcal R$-equivalent if $f_2=f_1\circ h^{-1}$ for some germ of a diffeomorphism $h$. 
The  simple singularities of germs of functions are classified by Arnold. Representatives of their $\mathcal R$-orbits when  $n=2$ are as follows 
 $$
 A_k: \pm(x^2\pm y^2),k\ge 2,\, D_k:x^2y\pm y^{k-1},k\ge 4,\, E_6:x^3+y^4,\, E_7:x^3+xy^3,\, E_8:x^3+y^5.
 $$

For $n\ge 3$ one adds the quadratic form $\pm x_3^2\pm \cdots\pm x_n^2$ to the above normal forms.

\section{Multiplicity of umbilic points}\label{sec:mult}
BDEs are studied extensively (see for example \cite{pairs} for a survey article). 
These are equations that can be written in the form
\begin{equation}\label{eq:genBDE}
	\omega: a(u,v)dv^2+b(u,v)dudv+c(u,v)du^2=0,
\end{equation}
where $a,b,c$ are smooth or analytic functions defined in some open set $U\subset \mathbb R^2$. 
We take $U$ to be a neighbourhood of the origin.

For analytic $a,b,c$, {\it the multiplicity} at the origin of BDE
(\ref{eq:genBDE}) is defined in \cite{mult} as the maximum number of 
folded singularities (see \S \ref{sec:appendix} for definition) that can appear in a deformation of the equation, and  is given by
$$
m(\omega)=\frac{1}{2}\dim_{\mathbb C} {{\mathcal O}_2}/{\langle\delta,a\delta_u^2-b\delta_u\delta_v+c\delta_v^2\rangle},
$$
where ${\mathcal O}_2$ denotes the ring of germs of holomorphic functions in two variables. 

The BDE of the lines of principal curvature of a surface in $\mathbb R^3_1$ (or $\mathbb R^3$) at a generic umbilic point 
is not stable when deformed in the set of all BDEs (see \cite{1pbde} for its 
bifurcations in generic 1-parameter families of BDEs and Figure \ref{fig:bifStar} for an example; the multiplicity of the BDE is equal to $3$). 
However, a generic umbilic point persists  when deforming the surface. 
We define here an invariant which counts the maximum number of generic umbilc points that can appear when deforming a surface at a non-generic umbilic point.

\begin{defn}\label{def:mult}
	Let $M$ be a smooth and regular surface in $\mathbb R^3_1$ {\rm (}or $\mathbb R^3${\rm )}
	and let ${\rm p} \in M$ be an umbilic point. {\rm The multiplicity of the umbilic} ${\rm p}$, denoted by  $m_{u}({\rm p})$,
is the maximum number of umbilc points it can split up into under deformations of the surface $M$.
\end{defn}

In the rest of the paper, we denote by 
$$
\mathfrak{a}=F\bar{n}-G\bar{m},\, \mathfrak{b}=E\bar{n}-G\bar{l},\ \mathfrak{c}=E\bar{m}-F\bar l
$$
the coefficients of the BDE of the lines of principal curvature (we use $l,m,n$ at points away from the $LD$). 

A point ${\rm p}={\bx}(q)$  is an umbilic point if, and only if, all the coefficients 
of equation (\ref{eq:principalLD}) vanish at $q$, that is, $\mathfrak{a}(q)=\mathfrak{b}(q)=\mathfrak{c}(q)=0$.
A crucial observation that allows us to compute $m_{u}({\rm p})$ is that we only need two of those coefficients to vanish 
for $ {\rm p}$ to be an umbilic point.

\begin{lem}\label{lem:2_cond_umb}
Let ${\rm p}={\bx}(q)\in M$ be an umbilic point. Then, 
\begin{itemize}
\item[{\rm (i)}] If $E(q) \neq 0$, $\mathfrak{b}(q) = \mathfrak{c}(q) = 0$ implies $\mathfrak{a}(q) = 0$.
\item[{\rm (ii)}] If $F(q) \neq 0$, $\mathfrak{a}(q) = \mathfrak{c}(q) = 0$ implies $\mathfrak{b}(q) = 0$.
\item[{\rm (iii)}] If $G(q) \neq 0$,  $\mathfrak{a}(q) = \mathfrak{b}(q) = 0$ implies  $\mathfrak{c}(q) = 0$.
\end{itemize}
\end{lem}

\begin{proof}
The proof is straightforward. 
For instance, for case (iii) we observe that  
$G(E\bar{m}-F\bar l)=-E(F\bar{n}-G\bar{m})+F(E\bar{n}-G\bar{l})$, so as $G(q)\ne 0$, 
$\mathfrak{a}(q)=\mathfrak{b}(q)=0$ imply   $\mathfrak{c}(q)=0$.
\end{proof}

\begin{lem}\label{lem:OneCoeffNotZero}
Let $M$ be a smooth surface in $\mathbb R^3_1$ and ${\rm p}$ a point on $M$. Then the coefficients of the first fundamental form $E,F,G$ cannot all vanish at $\rm p$. 
\end{lem}

\begin{proof}
As $M$ is a regular surface in $\mathbb R^3_1$, its tangent space at any point $\rm p$ contains a spacelike vector. We choose a local 
parametrisation ${\bx}$ of $M$, with ${\rm p}={\bx}(q)$, such that ${\bx}_u(q)$ is that vector. It follows that $E({\rm q})> 0$.
\end{proof}

Lemma \ref{lem:2_cond_umb} and Lemma \ref{lem:OneCoeffNotZero} allow us to use the results in \cite{arnold} on the multiplicity of holomorphic maps 
to compute $m_u({\rm p})$ when $M$ is analytic (see Remark \ref{rem:MultSmooth}). For holomorphic germs $h,k: \mathbb C^2,0\to \mathbb C,0$, we denote the codimension of the ideal 
${\langle h,k \rangle}$ in $\mathcal O_2$ by $m(h,k)$, so
$m(h,k)=\dim_{\mathbb C} {\mathcal O_2}/{\langle h,k \rangle}.$

\begin{prop}\label{prop:mult_formula}
Let $M$ be an analytic surface in $ \mathbb R^3_1$ and ${\rm p}$ an umbilic point on $M$.
Then the multiplicity of ${\rm p}$ is given by
$$
m_u({\rm p}) = \left\{ 
\begin{array}{lll}
 m(\mathfrak{b},\mathfrak{c})  & {\rm if}& E(q) \neq 0,\\
m(\mathfrak{a},\mathfrak{c})& {\rm if}& F(q) \neq 0,\\
m(\mathfrak{a},\mathfrak{b})& {\rm if}& G(q) \neq 0,
\end{array}
\right.
$$
where $\mathfrak{a},\mathfrak{b},\mathfrak{c}$ denote the complexifications of their real analytic counterpart $\mathfrak{a},\mathfrak{b},\mathfrak{c}$.
\end{prop}

\begin{proof}
By Lemma \ref{lem:OneCoeffNotZero} the coefficients $E,F,G$ 
cannot all vanish at any given point on $M$. It follows by Lemma \ref{lem:2_cond_umb} and the results in \cite{arnold} that the number of umbilic points 
concentrated at ${\rm p}={\bx}(q)$ is the multiplicity at $q$ of the ideal ${\langle \mathfrak{b},\mathfrak{c}\rangle}$ in case (i), 
 ${\langle \mathfrak{a},\mathfrak{c} \rangle}$ in case (ii) or  ${\langle \mathfrak{a},\mathfrak{b} \rangle}$ in case (iii).  
\end{proof}

\begin{rem}\label{rem:MultSmooth}
{\rm 
For a smooth surface, if say $G(q)\ne 0$ and the map-germ $H=(\mathfrak{a},\mathfrak{b}):\mathbb R^2,0\to \mathbb R^2,0$ is $l$-$\mathcal K$-finitely determined (see \cite{wall} for notation), we have 
$m(\mathfrak{a},\mathfrak{b})=m(j^l\mathfrak{a},j^l\mathfrak{b})$, where 
$j^lg$ denotes the $l$-jet of the function-germ $g$ at the origin, i.e., its Taylor polynomial of order $l$ without the constant term. We can then substitute $\mathfrak{a},\mathfrak{b}$ by the polynomial germs $j^l\mathfrak{a}$ and 
$j^l\mathfrak{b}$ and proceed as for the analytic case.
}
\end{rem}

\begin{prop}
The multiplicity of an umbilic point is invariant under local re-parametrisation of the surface and Lorentzian changes of coordinates in the ambient space $\mathbb R^3_1$.
\end{prop}

\begin{proof}
	Let $\bx'$ be another local parametrisation of $M$ with $\bx = \bx' \circ h$ for some local diffeomorphism $h$ of the plane. Then 
	$$\left(\begin{array}{c}
	\mathfrak{a}\\ \mathfrak{b}\\ \mathfrak{c}
	\end{array}\right) = \det \left(dh\right)^2 \left( \begin{array}{ccc}
		(\frac{\partial h_2}{\partial y})^2 & (\frac{\partial h_1}{\partial y})(\frac{\partial h_2}{\partial y}) & (\frac{\partial h_1}{\partial y})^2\\
		2 (\frac{\partial h_2}{\partial x}) (\frac{\partial h_2}{\partial y}) & (\frac{\partial h_1}{\partial y})(\frac{\partial h_2}{\partial x})+(\frac{\partial h_1}{\partial x})(\frac{\partial h_2}{\partial y}) & 2 (\frac{\partial h_1}{\partial x})(\frac{\partial h_1}{\partial y})\\
		(\frac{\partial h_2}{\partial x})^2 & (\frac{\partial h_1}{\partial x})(\frac{\partial h_2}{\partial x}) & (\frac{\partial h_1}{\partial x})^2
	\end{array}  
\right) 
\left(\begin{array}{c}
	\mathfrak{a}' \circ h\\\mathfrak{b}' \circ h\\\mathfrak{c}' \circ h
	\end{array}\right)$$
where $\mathfrak{a},\mathfrak{b},\mathfrak{c}$ and $\mathfrak{a}',\mathfrak{b}',\mathfrak{c}'$  are the coefficients of the equation of the lines of principal curvature of $M$ with respect to $\bx$ and $\bx'$, respectively. We conclude, using \mbox{Lemma \ref{lem:2_cond_umb}}, that 
$
m_u({\rm p})=\dim_{\mathbb C} {\mathcal O_2}/{\langle \mathfrak{a},\mathfrak{b},\mathfrak{c}  \rangle} = \dim_{\mathbb C}{\mathcal O_2}/{\langle \mathfrak{a}',\mathfrak{b}',\mathfrak{c}' \rangle}.
$

The multiplicity of an umbilic point is also invariant under Lorentzian changes of coordinates in the ambient space as the equation of the lines of principal curvature remains unchanged 
by these transformations.
\end{proof}

\begin{theo}\label{theo:mult&multUmb}
Let ${\rm p}\in M\subset \mathbb R^3_1$ be  umbilic point and suppose that $m_u({\rm p})$ and $m(\omega_P)$ are finite. Then $ m(\omega_P)\ge 3m_u({\rm p})$.
\end{theo}

\begin{proof}
Intuitively, by deforming the surface, we deform its BDE of the lines of principal curvature. 
Umbilic points are singular points of that BDE and we get a maximum $m_u({\rm p})$ of them. 
As BDE $\omega_{P}$  has multiplicity at least $3$ at umbilic points, 
it should follow that $m(\omega_P)\ge 3m_u({\rm p})$.

Suppose that ${\rm p}$ is a timelike umbilic point. Then we can choose a coordinate system so that the coordinates
curves are lightlike curves (Theorem 3.1 in \cite{IzumiyaTari}), that is, $E\equiv 0$ and $G\equiv 0$ ($F\ne 0$ as the surface is regular). 
Then equation (\ref{eq:principalBDE}) becomes  $(\omega_P): ndv^2-ldu^2=0$ and has 
discriminant function $\delta=4ln$. Then,
$$
\begin{array}{rcl}
	m(\omega_P)&=&\frac{1}{2}m(\delta,\mathfrak{a}\delta_u^2-\mathfrak{b}\delta_u\delta_v+\mathfrak{c}\delta_v^2)\\
	&=&\frac{1}{2}m(ln,n(l_un+ln_u)^2-l(l_vn+ln_v)^2)\\
	&=&\frac{1}{2}m(ln,l_u^2n^3-n_v^2l^3)\\
	&=&\frac{1}{2}(m(l,l_u^2n^3-n_v^2l^3)+m(n,l_u^2n^3-n_v^2l^3))\\
	&=&\frac{1}{2}(6m(l,n)+2m(l,l_u)+2m(n,n_v))\\
	&=&3m_u({\rm p})+m(l,l_u)+m(n,n_v).
\end{array}
$$
It follows that  $m(\omega_P)\ge 3m_u({\rm p})$. 

Suppose now that ${\rm p}$ is a spacelike or lightlike umbilic point. 
If ${\rm p}$ is spacelike, then $E(q)\ne 0$ and $G(q)\ne 0$.
If ${\rm p}$ is a lightlike umbilic point, we cannot have $E(q)=G(q)=0$ as $M$ is a regular surface in $\mathbb R^3_1$ so $T_{{\rm p}}M$ contains a unique lightlike direction. Hence, either $E(q)\ne 0$ or $G(q)\ne 0$. 
We suppose, without loss generality,
that $G(q)\ne 0$. 
It follows that $\mathfrak{c}=(\mathfrak{b}F-\mathfrak{a}E)/G$ (see the proof of Lemma \ref{lem:2_cond_umb}), so 
$$
\begin{array}{rcl}
 m(\omega_P)&=&\frac{1}{2}m(\delta,\mathfrak{a}\delta_u^2-\mathfrak{b}\delta_u\delta_v+\mathfrak{c}\delta_v^2)\\
 &=&\frac{1}{2}m(\mathfrak{b}^2-4\mathfrak{a}\mathfrak{c},\alpha_1 \mathfrak{a}^3+\alpha_2\mathfrak{a}^2\mathfrak{b}+\alpha_3\mathfrak{a}\mathfrak{b}^2+\alpha_4\mathfrak{b}^3)\\
 &=&\frac{1}{2}m(\mathfrak{b}^2-4\mathfrak{a}\mathfrak{c},\alpha_1 \mathfrak{a}^3+\alpha_2\mathfrak{a}^2\mathfrak{b}+4\alpha_3\mathfrak{a}^2\mathfrak{c}+4\alpha_4\mathfrak{a}\mathfrak{b}\mathfrak{c})\\ 
 &=&\frac{1}{2}m(\mathfrak{b}^2-4\mathfrak{a}\mathfrak{c},\mathfrak{a})+\frac{1}{2}m(\mathfrak{b}^2-4\mathfrak{a}\mathfrak{c},\alpha_1 \mathfrak{a}^2+\alpha_2\mathfrak{a}\mathfrak{b}+4\alpha_3\mathfrak{a}\mathfrak{c}+4\alpha_4\mathfrak{b}\mathfrak{c})\\
 &=&m(\mathfrak{a},\mathfrak{b})+\frac{1}{2}m(\mathfrak{b}^2-4\frac{F}{G}\mathfrak{a}\mathfrak{b}+4\frac{E}{G}\mathfrak{a}^2,\beta_1\mathfrak{a}^2+\beta_2\mathfrak{a}\mathfrak{b}+\beta_3\mathfrak{b}^2)\\
 &=&m(\mathfrak{a},\mathfrak{b})+\frac{1}{2}m(\mathfrak{b}^2-4\frac{F}{G}\mathfrak{a}\mathfrak{b}+4\frac{E}{G}\mathfrak{a}^2,{\beta}_7\mathfrak{a}^2+{\beta}_8\mathfrak{a}\mathfrak{b})\\
 & =&2m(\mathfrak{a},\mathfrak{b})+\frac{1}{2}m(\mathfrak{b}^2-4\frac{F}{G}\mathfrak{a}\mathfrak{b}+4\frac{E}{G}\mathfrak{a}^2,{\beta}_7\mathfrak{a}+{\beta}_8\mathfrak{b}),
\end{array}
$$
where $\alpha_i$ and $\beta_i$ are functions depending on $\mathfrak{a},\mathfrak{b},\mathfrak{c}$ and their partial derivatives. As the ideal $I=\langle \mathfrak{b}^2-4\frac{F}{G}\mathfrak{a}\mathfrak{b}+4\frac{E}{G}\mathfrak{a}^2,{\beta}_7\mathfrak{a}+{\beta}_8\mathfrak{b} \rangle$ is contained in the ideal $J=\langle \mathfrak{b}^2, \frac{F}{G}\mathfrak{a}\mathfrak{b}, \frac{E}{G}\mathfrak{a}^2,{\beta}_7\mathfrak{a}+{\beta}_8\mathfrak{b} \rangle$, it follows that
$$
m(I)\ge m(J)\ge 2m(\mathfrak{b},\frac{F}{G}\mathfrak{a}\mathfrak{b},\frac{E}{G}\mathfrak{a}^2,{\beta}_7\mathfrak{a}+{\beta}_8\mathfrak{b})\nonumber\\
=2m(\mathfrak{b},\frac{E}{G}\mathfrak{a}^2,{\beta}_7\mathfrak{a})
\ge  2m(\mathfrak{a},\mathfrak{b}).
$$
Consequently, $m(\omega_P)\ge 3m_u({\rm p})$.
\end{proof}

\subsection{Umbilic points of multiplicity $1$ and deformation of degenerate umbilics }\label{sunsec:mult1}

We characterise now  umbilic points of multiplicity $1$. 

In \cite{bruce84}, Bruce described a technique for studying generic properties of surfaces 
in the Euclidean space $\mathbb R^3$. 
Locally at each point ${\rm p}$ on a compact orientable surface in $\mathbb R^3$
is chosen an orthonormal frame $\{ {\bf u}, {\bf v}, {\bf n}\}$ with $ {\bf n}$ normal to $M$, 
so that $M$ is locally the graph of a function $z=f_{{\rm p}}(x,y)$, with $j^1f_{{\rm p}}(0,0)=0$. One can then construct the Monge-Taylor map $\theta: M\to \mathcal V\subset J^k(2,1)$, given by ${{\rm p}}\mapsto j^kf_{{\rm p}}(0,0)$, with $\mathcal V$ consisting of polynomials of 
 $  2\le $ degree $\le k$.
The idea is that the $k$-jet of $f_{{\rm p}}$ at the origin contains all the sought after geometric properties of $M$ at ${{\rm p}}$.

For a surface $M$ in $\mathbb R^3_1$, we can follow Bruce's approach in the spacelike or timelike regions of $M$ 
by choosing a Lorentzian orthonormal system of coordinates at each point in these regions. At a spacelike umbilic point and taking 
the surface as the graph of a function $z=f(x,y)$, Proposition 2 in \cite{bruce84} applies and we have 
the tangent space to the image of the Monge-Taylor map spanned by the vectors $v_1$ and $v_2$ with

{\footnotesize
$$
\begin{array}{rcl}
v_1&=& j^k(-f_{xx}(0,0)x -f_{xy}(0, 0)y + f_x(x, y) -f_{xx}(0,0)f_x(x,y)f(x,y)-f_{xy}(0,0)f_y(x,y)f(x,y)),\\
v_2&=&j^k(-f_{xy}(0, 0)x -f_{yy}(0, 0)y + f_y(x, y) -f_{xy}(0,0)f_x(x,y)f(x,y)-f_{yy}(0,0)f_y(x,y)f(x,y)).
\end{array}
$$
}

We have a similar result when the surface is a Lorentzian patch given as the graph of $y=f(x,z)$.

\begin{prop} \label{prop:dtheta}
Let $M$ be a Lorentzian patch in $\mathbb R^3_1$ given as the graph of a function $y=f(x,z)$ with $(x,y)$ near the origin, and 
let $\theta :M \to \mathcal V$ be the Monge-Taylor map. Then the tangent space at $\theta(0,0)$ 
to the image of $\theta$ is spanned by
{\footnotesize
	$$
\begin{array}{rcl}
	v_1&=&j^k(-f_{xx}(0,0)x -f_{xz}(0, 0)z + f_x(x, y) -f_{xx}(0,0)f_x(x,y)f(x,z)+f_{xz}(0,0)f_z(x,y)f(x,y)),\\
	v_2&=&j^k(-f_{xz}(0, 0)x -f_{zz}(0, 0)z + f_z(x, z) +f_{xz}(0,0)f_x(x,z)f(x,z)+f_{zz}(0,0)f_z(x,y)f(x,z)).
\end{array}
	$$
}
\end{prop}

\begin{proof}
The details of the proof of Proposition 2 in \cite{bruce84} for surfaces in the Euclidean space 
were omitted in \cite{bruce84} and are given in \cite{joey}. 
The proof for when $M$ is a Lorentzian patch follows the same steps as those in \cite{joey} and we omit it. 
\end{proof}

At points on the $LD$ the normal to the surface is tangent to the surface, so we cannot choose 
an orthonormal frame as in \cite{bruce84}. As our study is local, we choose a fixed 
coordinate system  $\{ {\bf u}, {\bf v}, {\bf w}\}$ in a neighbourhood $W$ of a point ${\rm p}_0$ on the $LD$
with ${\bf w}$ transverse to $M$ at points in $W$. Then $M$ can be parametrised locally at each point ${\rm p}\in W$ by $(x,y)\mapsto \bx_{{\rm p}}(x,y)=(x,y,f_{{\rm p}}(x,y))$.
We define the Monge-Taylor map 
$\theta: W\to J^k(2,1)$ by 
$
\theta({{\rm p}})=j^kf_{{\rm p}}(0,0).
$

We write the $k$-jet, at the origin, of a function $g$ in the 
form $j^kg=\displaystyle{\sum_{s=1}^k\sum_{i=0}^s }a_{si}u^{s-i}v^i$ and identify it with the coefficients $(a_{si})$.

We call  {\it umbilic stratum} the set $\mathcal U$ in $\mathcal V$ (or $J^k(2,1)$ for lightlike umbilic points) such that ${\rm p}$ is an umbilic point 
if and only if $\theta({\rm p})\in \mathcal U$. 

\begin{theo}\label{theo:mult1_Trans_Jetspave}
An umbilic point on an analytic and regular surface in $\mathbb R^3_1$ has multiplicity $1$ if, and only if, the Monge-Taylor map is transverse to the umbilic stratum. 
\end{theo}

\begin{proof} 
We prove the result when ${\rm p}$ is a timelike umbilic point, the other two cases follow similarly.
We work in $\mathcal V\subset J^3(2,1)$ and take the surface as the graph of a function $y=f(x,z)$ with
$$
j^3 f= 
\frac{\kappa}{2} (x^2 - z^2)+a_{30} x^3+a_{31} x^2 z+a_{32} x z^2+a_{33} z^3
$$
and $(x,z)$ in a neighbourhood $U$ of the origin. 
The umbilic stratum in $\mathcal V$ is given by 
$$\mathcal U: 
\left\{ 
\begin{array}{c}
	a_{21} = 0,\\
a_{20}+a_{22} =0.
\end{array}
\right.
$$
and its tangent space at $\theta(0,0)$ is the intersection of the kernels of the 
	1-forms $\eta_1=da_{21}$ and $\eta_2=da_{20}+da_{22}$.
By Proposition \ref{prop:dtheta}, the tangent space to the image of $d\theta$ at $\theta(0,0)$ is spanned by
$v_1$ and $v_2$ with 
$$
\begin{array}{rcl}
	v_1&=&3a_{30}x^2+2a_{31}xz+a_{32}z^2+O(3),\\
	v_2&=&a_{31}x^2+2a_{32}xz+3a_{33}z^2+O(3),
\end{array}
$$
where $O(l)$ denotes a reminder of order $l$ in $(x,z)$. 
The map $\theta$ is transverse to $\mathcal U$ at $\theta(p)$ if, and only if, there are no 
scalars $\lambda$ and $\mu$ such that 
$\eta_1(\lambda v_1+\mu v_2)=\eta_2(\lambda v_1+\mu v_2)=0$ and $\lambda\ne 0 $ or $\mu\ne 0$, equivalently,
$
\eta_1(v_1).\eta_2(v_2)-\eta_1(v_1).\eta_2(v_2)\ne 0.
$
This occurs if, and only if,
$$
a_{32}^2 - 3 a_{33} a_{31} - a_{31}^2 + 3 a_{32} a_{30} \neq 0.
$$

By Proposition \ref{prop:mult_formula}, $
m_u({\rm p}_0) = \dim_\mathbb{C} {\mathcal O_2}/{\langle \mathfrak{b},\mathfrak{c} \rangle}$ 
as $E(x,y)=1+f_x^2(x,y)\ne 0$. We have 
$\mathfrak{b}(x,y)=(a_{32} +3 a_{30}) x+ (3 a_{33} + a_{31}) z+O(2)$, $\mathfrak{c}(x,y)=a_{31} x+ a_{32} z + O(2) $, so 
$m_u({\rm p}_0)=1$ if, and only if,
$
a_{32}^2 - 3 a_{33} a_{31} - a_{31}^2 + 3 a_{32} a_{30} \neq 0,
$
equivalently, $\theta$ is transverse to $\mathcal{U}$ 
at $\theta (0,0)$.
\end{proof}

\begin{theo}\label{theo:defm}
	Let $M$ be a regular analytic surface in $ \mathbb R^3_1$, and let ${\rm p}_0 \in M$ be an umbilic point. 
	Then there exists a $1$-parameter family of analytic surfaces $M_t$ with $M_0=M$ 
	such that all the umbilic points of $M_t$ near ${\rm p}_0$ have multiplicity $1$ for all $t\ne 0$ small enough.
\end{theo}

\begin{proof}
We give the proof for the case when ${\rm p}_0$ is a timelike umbilic point as the expressions of the strata involved are simpler 
and refer to \cite{MarcoThesis} for the cases when ${\rm p}_0$ is a spacelike or a lightlike umbilic. 

Following the notation in the proof of Theorem \ref{theo:mult1_Trans_Jetspave}, 
the stratum of umbilic points of multiplicity $>1$ is given by 
	$$
\mathcal W:	\left\{\begin{array}{c}
a_{21}=0,\\
a_{20}+a_{22} =0,\\
a_{32}^2 - 3 a_{33} a_{31} - a_{31}^2 + 3 a_{32} a_{30} =0.
\end{array}
\right.
$$

The stratum $\mathcal W$ is a codimension $3$ variety in $\mathcal V$ and 
 is singular if, and only if, $a_{30}=a_{31}=a_{32}=a_{33}=0.$ 
At such points, and in $\mathcal V\subset J^3(2,1)$, it is diffeomorphic to a cone 
(it is diffeomorphic to the zero set of the map $h:\mathbb R^7\to \mathbb R^3$, 
given by $h((a_{ij}))=(a_{21},a_{20},- a_{30}^2- a_{31}^2+a_{32}^2 +a_{33}^2 )$).

We take the surface $M=M_0$  as the graph of $y=f_0(x,z)$, 
with  $f_0(x,z)=\frac{\kappa_0}{2}(x^2-z^2)+C_0(x,z)+O_4(x,z)$, $C_0(x,z)=a_{30}^0x^3+a_{31}^0x^2z+a_{32}^0xz^2+a_{33}^0z^3$
and $(x,z)$ in a neighbourhood $V_1$ of the origin 
in $\mathbb R^2$.

When $C_0$ is identically zero (the other case is similar), $\theta(0,0)$ is a singular point of $\mathcal W$ 
and the tangent space to the image of the Monge-Taylor map 
at $\theta(0,0)$ is contained in $T_{\theta(0,0)}\mathcal U$ (this is not the case when $C_0\ne 0$).
Consider the family of surfaces $y=f_0(x,z)+Q(x,z)+C(x,z)$
with the coefficients of $Q(x,z)=a_{20}x^2+a_{21}xz$ and of $C(x,z)=a_{30}x^3+a_{31}x^2z+a_{32}xz^2+a_{33}z^3$ varying in a small enough 
neighbourhood $V_2$ of the origin in $\mathbb R^2\times \mathbb R^4=\mathbb R^6$. 
This yields a family of Monge-Taylor maps $\tilde \theta: V_1\times V_2\to \mathcal V$ 
which is transverse to $\mathcal W$ at $\theta(0,0)$ ($\tilde \theta$ is a submersion). 
Then $\tilde \theta^{-1}(\mathcal W)$ is a germ of a codimension 3 variety of $ V_1\times V_2$
and is diffeomorphic to a cone, so we can choose a path
	$\Gamma: t\to \gamma(t)\subset V_2$ so that the image of  $V_1\times \Gamma$ by $\tilde \theta$  
avoids $\mathcal W$ except at $\theta(0,0)$. 
This means that all the umbilic points of the surface $M_t$ near ${\rm p}_0$, with $t\ne 0$, have multiplicity equal to $1$.

At a lightlike umbilic points we consider the Monge-Taylor map with a fixed frame and with the umbilic stratum given by
$$\mathcal U: 
\left\{ 
\begin{array}{rcl}
	(a_{01}^2-1)a_{21}-2a_{10} a_{01} a_{22} & = & 0,\\
	(a_{01}^2-1)a_{20}-(a_{10}^2-1)a_{22} & = & 0, \\
	(a_{10}^2-1)a_{21} - 2 a_{10} a_{01} a_{20} & = & 0.
\end{array}
\right.
$$

The above equations define a codimension 2 variety in the jet space $J^k(2,1)$ (one of the equations is redundant by Lemma \ref{lem:2_cond_umb}), for $k\ge 2$.
The equation of umbilics of multiplicity $>1$ is too lengthy to reproduce here (see \cite{MarcoThesis} for details). 
We prove in \cite{MarcoThesis} that $\mathcal W$ is diffeomorphic to  a cone and proceed as above.
\end{proof}

\begin{theo}\label{theo:mult_1_lpl}
{\rm (1)} Let ${\rm p}\in M$ be a spacelike or timelike umbilic point. Then $m_u({\rm p})=1$ if, and only if, 
 the discriminant of BDE $(\ref{eq:principalLD})$ of the lines of principal curvature  {\rm (}which is the 
 $LPL$ when $\rm p$ is a timelike umbilic{\rm )} has a Morse singularity.
	
{\rm (2)} Let ${\rm p}\in M$ be a lightlike  umbilic point. Then $m_u({\rm p})=1$ if, and only if, the $LD$ has a Morse singularity.
\end{theo}

\begin{proof} We take $M$ at ${\rm p}$ as the graph of a function $f$. 
	If the cubic part of $f$ is identically zero, then $j^1\omega_P=0$ which implies $m_u({\rm p})>1$.	
	We can suppose then that the cubic part of $f$ is not identically zero and take $f$ as in  \S\ref{sec:appendix}.
	
(1) When ${\rm p}$ is a spacelike umbilic point, and using (\ref{eq:1jetBDEspace}) in \S\ref{sec:appendix}, we have $m_u({\rm p})=m(\mathfrak{a},\mathfrak{b})=1$ if, and only if, $s^2+t^2\ne 9$.
The circle $s^2+t^2= 9$ is where the discriminant of the BDE of the lines of principal curvature 
has a singularity more degenerate than Morse (Proposition \ref{prop:spacelikeumbPart}; it is also where 
$\phi$ and $\alpha$ have a common root).

When ${\rm p}$ is a timelike umbilic point, using the 1-jets in Proposition \ref{prop:1jet_bde_timelike_umb}in \S\ref{sec:appendix}, we have 
for case (i) (resp. case (iii)): $m_u({\rm p})=m(\mathfrak{a},\mathfrak{b})=1$ if, and only if, $s^2-t^2-3t\ne 0$ (resp. $(t-1)(1+t\pm s)\ne 0$). 
The result follows by Proposition \ref{prop:1jet_bde_timelike_umb}.

(3)	When ${\rm p}$ is a lightlike umbilic point, we parametrise $M$ locally at ${\rm p}$ as in (\ref{paramLightUmb}) in \S\ref{sec:appendix}.
Using the 1-jet in Proposition \ref{prop:Lightlikeumbilic}(i), 
we have $m_u({\rm p})=m(\mathfrak{a},\mathfrak{b})=1$ if, and only if, $6a_{22}^2a_{30}\pm 3a_{30}a_{32}\mp a_{31}^2\ne 0$, which is precisely the 
condition for the $LD$ to have a Morse singularity (Proposition \ref{prop:Lightlikeumbilic}(iii); observe that in this case the discriminant of the BDE of the lines of principal curvature has an $A_3$-singularity at a generic lightlike umbilic). 
\end{proof}

The condition in Theorem \ref{theo:mult_1_lpl} are intimately related to the singularities of the distance squared function on $M$.
The family of distance squared functions $d^2: M \times \mathbb R^3_1 \to \mathbb R$ on $M$ is given by
$
d^2({\rm p},{\bf v}) = \langle {\rm p}-{\bf v}, {\rm p}-{\bf v} \rangle.
$
For ${\bf v} \in \R^3_1$,  we denote by $d_{\bf v}^2: M \to \mathbb R$ the distance squared function from 
$\bf v$, given by $d_{\bf v}^2({\rm p}) = d^2({\rm p},{\bf v})$. The singularities of $d_{\bf v}^2$ measure the contact of $M$ 
with pseudo-spheres of centre ${\bf v}$.
A point ${\rm p}\in M$ is an umbilic point if, and only if, $d_{\bf v}^2$ 
has a singularity of corank 2, that is $j^2 d_{\bf v}^2=0$ at ${\rm p}$ (\cite{causticsR31}). 
At generic umbilics, the singularity is of type $D_4^{\pm}$.

\begin{theo}\label{teo:mult_versal}
Let ${\rm p}$ be an umbilic point on a regular analytic surface $M\subset  \mathbb R^3_1$ 
and suppose that $d_{{\bf v}_0}^2$ has a  $D_4^{\pm}$-singularity at ${\rm p}$ for some ${\bf v}_0 \in \mathbb R^3_1$. Then $m_u({\rm p}) = 1$ if, and only if, the family of distance squared function $d^2$ is an
$\mathcal{R}^+$-versal deformation of the singularity  of $d_{{\bf v}_0}^2$ at ${\rm p}$.
\end{theo}

\begin{proof}
We take the surface as the graph of a function $f$ with ${\rm p}$ the origin (see \S\ref{ssec:config}).  
At a timelike or spacelike umbilic point 
we have 
$j^3d_{{\bf v}_0}^2\sim_{ \mathcal R} C(x,y)$, with $C$ the cubic part of $j^3f$. 
Then  the singularity of $d_{{\bf v}_0}^2$ at the origin, with ${\bf v}_0=(0,\kappa/2,0)$, is a $D_4^{\pm}$ if, and only if $C$ has no repeated roots 
(in particular, $C$ cannot be identically zero).

Suppose that ${\rm p}$ is a spacelike umbilic point and take $C=\Re(z^3+\beta z^2\bar{z})$ as in \S\ref{ssec:config}. 
The origin is a $D_4^{\pm}$-singularity of $d_{{\bf v}_0}^2$ means that $C$ is not on the curve 
$$
-27+18 s^2 -8s^3 +s^4 +18t^2 +24st^2 +2s^2 t^2 +t^4 = 0
$$ 
(the inner hypocycloid in Figure \ref{fig:PartitionA_1^+}). 
The family of distance squared function 
is an $\mathcal{R}^+$-versal deformation of the $D_4^{\pm}$-singularity of $d_{{\bf v}_0}^2$ at the origin 
if, and only if, $\beta$ is not on the circle $s^2+t^2=9$ (\cite{porteous}). The result then follows by 
Theorem \ref{theo:mult_1_lpl}(1) and Proposition \ref{prop:spacelikeumbPart}.

If ${\rm p}$ is a timelike umbilic point, we take $C$ as in Proposition \ref{prop:cubic}(i) or (iii).
For $C(x,z)=x(x^2+sxz+tz^2)$, we need $t(s^2-4t)\ne 0$ for $d^2_{{\bf v}_0}$ to have a $D_4^{\pm }$-singularity. 
We write ${\bf v}=(a,b+\kappa/2,c)$ for ${\bf v}$ near ${\bf v}_0$ and denote by 
$\dot{d}^2_a(x,z)=\frac{\partial d^2}{\partial a}(x,z,{\bf v}_0)$ and similarly for $\dot{d}^2_b$ and $\dot{d}^2_c$. Then  
the family of distance squared function is an $\mathcal R^+$-versal deformation of the $D_4^{\pm }$-singularity of $d^2_{{\bf v}_0}$ if, 
and only if, 
$$
\mathcal O_2
\langle \frac{\partial d^2_{{\bf v}_0}}{\partial x}, \frac{\partial d^2_{{\bf v}_0}}{\partial z} \rangle + \mathbb R.\{1, \dot{d}^2_a,\dot{d}^2_b,\dot{d}^2_c\} =\mathcal O_2.
$$ 

As a $D_4^{\pm }$-singularity is 3-$\mathcal R$-determined, it is enough to show that the above equality holds 
at the 3-jet level. 
That occurs if, and only if, $s^2-t^2-3t\ne 0$.
The result then follows by Theorem \ref{theo:mult_1_lpl}(1) and 
 Proposition \ref{prop:1jet_bde_timelike_umb}.

When $C$ is as in Proposition \ref{prop:cubic}(iii), similar calculations to the above show that 
we have an $\mathcal R^+$-versal family if, and only if, $(t-1)(1+t\mp s)\ne 0$. The result then follows by Theorem \ref{theo:mult_1_lpl}(1) and 
Proposition \ref{prop:1jet_bde_timelike_umb}.

At a lightlike umbilic point and with the parametrisation of the surface as in (\ref{paramLightUmb}),
 $d_{{\bf v}_0}^2$ has a singularity at the origin 
when ${\bf v}_0=(-\frac{1}{2a_{22}},0,-\frac{1}{2a_{22}})$. Then 
$$j^3d_{{\bf v}_0}^2(x,y,c)=
-\frac1{a_{22}}
(a_{30}x^3+a_{31}x^2y+(2a_{22}^2+a_{32})xy^2+a_{33}y^3)
$$
and the singularity is of type $D_4^{\pm}$ when the above cubic form has no repeated roots, that is, 
$$
\begin{array}{l}
32a_{22}^6a_{30}+4(12a_{30}a_{32}-a_{31}^2)a_{22}^4-4(9a_{30}a_{31}a_{33}-6a_{30}a_{32}^2+a_{31}^2a_{32})a_{22}^2\\
+27a_{30}^2a_{33}^2-18a_{30}a_{31}a_{32}a_{33}+4a_{30}a_{32}^3+4a_{31}^3a_{33}-a_{31}^2a_{32}^2\ne 0.
\end{array}
$$

A calculation similar to that for a timelike umbilic point shows that 
the family $d^2$ is an versal $\mathcal R^+$-versal deformation of the $D^{\pm}_4$-singularity of  $d_{{\bf v}_0}^2$ 
if, and only if, 
$6a_{22}^2a_{30}\pm 3a_{30}a_{32}\mp a_{31}^2\ne 0$. The result follows by Theorem \ref{theo:mult_1_lpl}(2) and 
Proposition \ref{prop:Lightlikeumbilic}.
\end{proof}

\begin{rem}\label{rem:Mult1}
{\rm The multiplicity of an umbilic point ${\rm p}$ can still be equal to $1$ even when the singularity of   
$d_{{\bf v}_0}^2$ is more degenerate than $D_4^{\pm}$. At a spacelike umbilic, this is the case for $C$ on the 
inner hypocycloid in Figure \ref{fig:PartitionA_1^+} and away from the points on the circle $s^2+t^2=9$.
}
\end{rem}

\subsection{Umbilic points of multiplicity greater than $1$}

We relate here the multiplicity of umbilic points to the singularities of the $LD$ and of $LPL$.

\begin{theo}\label{theo:LDMilnor}
Let ${\rm p}\in M$ be a lightlike umbilic point. 
If the $LD$ has a singularity which $\mathcal R$-equivalent to a quasi-homogeneous singularity, then $m_u({\rm p})$ is equal to the 
Milnor number of the singularity of the $LD$ at ${\rm p}$.
\end{theo}

\begin{proof}
We take $M$ locally at ${\rm p}$ as the graph of a function $z=f(x,y)$ with $(x,y)$ in some neighbourhood of the origin. Then the $LD$ is given by $\delta = 0$, where
$$
\delta(x,y) = f_x(x,y)^2+f_y(x,y)^2-1.
$$

We have $E(0,0)=1-f^2_x(0,0)\ne 0$ or 
$G(0,0)=1-f^2_y(0,0)\ne 0$. Suppose, without loss of generality, that $f^2_y(0,0) \neq 1$. 
We have $\mathfrak{a}=f_{xy} (1 - f_{y}^2) + f_{x} f_{y} f_{yy}$ and $\mathfrak{b}=f_{xx} (1 - f_{y}^2) - f_{yy} (1 - f_{x}^2)$.
It follows from Proposition \ref{prop:mult_formula} and from the fact that $f_x(0,0)\ne 0$ 
(otherwise $f^2_y(0,0)=1$ as the origin is on the $LD$) that
$$
\begin{array}{rcl}
	m_u({\rm p}) & = & m(\mathfrak{a},\mathfrak{b})\\
	& = & m(f_{xy} (f_{x}^2 - \delta) + f_{x} f_{y} f_{yy}, f_{xx} (f_{x}^2 - \delta) - f_{yy} (f_{y}^2 - \delta)) \\
	& = & m(\frac12{f_x} \delta_y - f_{xy} \delta, \frac12{f_x} \delta_x - \frac12{f_y} \delta_y + (f_{yy} - f_{xx}) \delta)\\
	& = & m (\delta_y - \frac{2 f_{xy}\delta}{f_x} , \frac12{f_x} \delta_x - \frac12{f_y} \delta_y + (f_{yy} - f_{xx}) \delta)\\
	& = & m(\delta_y - \frac{2 f_{xy}\delta}{f_x} , \frac12{f_x} \delta_x + (f_{yy} - \frac{f_{xy} f_{y}}{f_{x}} - f_{xx}) \delta) \\
	& = & m (\delta_y - \frac{2 f_{xy}\delta}{f_x}, \delta_x - \frac{2 (f_{x} f_{xx} + f_{xy} f_{y} - f_{x} f_{yy})\delta}{f_{x}^2} )\\
	&=&m(\delta_x + f_1 \delta,\delta_y + f_2 \delta),
\end{array}
$$
with $f_1 = - {2 (f_{x} f_{xx} + f_{xy} f_{y} - f_{x} f_{yy})}/{f_{x}^2}$ and $f_2 = - {2 f_{xy}}/{f_x}$.
From the hypothesis, there is a germ of a diffeomorphism $h=(h_1,h_2)$ such that $P = \delta\circ h$ is a quasi-homogeneous polynomial, so
$P(u,v) = r_1{u}P_u(u,v)+r_2{v} P_v(u,v)$
for some strictly positive rational numbers $r_1$ and $r_2$. 
Using the facts that $P_u=	(\delta_x \circ h) \frac{\partial h_1}{\partial u} + (\delta_y \circ h) \frac{\partial h_2}{\partial u}$ 
$P_v=(\delta_x \circ h) \frac{\partial h_1}{\partial v} + (\delta_y \circ h) \frac{\partial h_2}{\partial v}$ and $\det dh(0,0) \neq 0$, we get 
$$
\begin{array}{rcl}
	m_u({\rm p}) & = & m ( \delta_x \circ h + (f_1 \circ h) (\delta \circ h), \delta_y \circ h + (f_2 \circ h) (\delta \circ h)) \\
& = & m(  \delta_x \circ h + f_3 P, \delta_y \circ h + f_4 P) \\
& = & m( dh (\delta_x \circ h + f_3 P, \delta_y \circ h + f_4 P))\\
	& = & m( \frac{\partial h_1}{\partial u} (\delta_x \circ h) + \frac{\partial h_2}{\partial u} (\delta_y \circ h) + f_5 P, \frac{\partial h_1}{\partial v} (\delta_x \circ h) + \frac{\partial h_2}{\partial v} (\delta_y \circ h) + f_6 P  ) \\
	& = & m ( P_u + f_5 P, P_v + f_6 P)  \\
	& = & m ( (1+ r_1uf_5) P_u + r_2vP_vf_5 , (1+ r_2vf_6) P_v +  r_1{u} P_u  f_6) \\
	& = & m ( P_u + f_7 P_v, P_v + f_8 P_u) \\
	& = & m ( P_u, P_v )=\mu(P)=\mu(\delta),
\end{array}
$$
where
$f_3=f_1 \circ h$, $f_4=f_2 \circ h$, 
$f_5 = \frac{\partial h_1}{\partial u} f_3 + \frac{\partial h_2}{\partial u} f_4,$ $f_6 = \frac{\partial h_1}{\partial v} f_3 + \frac{\partial h_2}{\partial v} f_4$, $f_7 = {r_2vf_5 }/{(1+ r_1uf_5)},$ $f_8 = {r_1uf_6 u}/({1+ r_2vf_6)},$ and where $\mu(\delta)$ denotes the Milnor number of $\delta$ at the origin.
\end{proof}

\begin{theo}\label{theo:sing_umb_time}
{\rm (1)} Let ${\rm p}\in M$ be a spacelike umbilic point. The possible simple singularities of the discriminant 
of the equation of the lines of principal curvature at ${\rm p}$ are $A_{2k+1}^+$ and these can be realised on surfaces in $\mathbb R^3_1$.

{\rm (2)} Let ${\rm p}\in M$ be a timelike umbilic point.
 The possible simple singularities of the $LPL$ at ${\rm p}$ 
are  $A_{2 k+1}^-$, $D^{\pm}_k$ and $E_7$,  
and these can be realised on the $LPL$ of surfaces in $\mathbb R^3_1$.
\end{theo}

\begin{proof}
(1) In the Riemannian region of $M$ all points on the discriminant of the BDE of lines of principal curvature (these are spacelike umbilic points) are singular points. Therefore, the only possible simple singularities of the discriminant are $A_{2k+1}^+$. These singularities can be realized at the origin on the graph of $z=f(x,y)=x^3-x y^{k+2}$. (The expression of the discriminant is lengthy to reproduce here.)

(2) We take a local parametrisation ${\bx}$ of $M$ with ${\bx}(0,0)={\rm p}$ and 
such that the coordinate curves are lightlike, i.e., $E\equiv0\equiv G$ 
(by Theorem 3.1 in \cite{IzumiyaTari}). Then the $LPL$ is the zero set of 
$\delta(u,v)=l(u,v) n(u,v)$, with $l(0,0) = n(0,0) = 0$.
As the singularities $A_{2k}$, $E_6$ and $E_8$ are irreducible, 
they cannot occur on the $LPL$ at timelike umbilic points. 

For realising the simple singularities of the $LPL$, a possible approach is to apply the fundamental theorem of Lorentzian surfaces in $\mathbb R^3_1$ (see for example  \cite{Couto_Lymb})
with $E\equiv0\equiv G$ and $m\equiv0$. 
This yield a system of partial differential equations in $F,l,n$. Here we only need
some particular solutions of the system and take the surfaces at the origin as the graph of some functions $y=f(x,z)$ with

$f(x,z)=x^3+x z^{k+2}$ for an $A^-_{2k+1}$-singularity of the $LPL$;

$f(x,z)=(x^2-z^2) + (x-z) (x+z)^2 \pm (x-z)^k$  for a $D_{k}^\pm$-singularity of the $LPL$;

$f(x,z)=(x^2-z^2)+(x-z)^3+(x-z) (x+z)^3$ for a  $E_7$-singularity of the $LPL$.
\end{proof}

\begin{theo}\label{theo:LPLSimple}
{\rm (1)} Let ${\rm p}\in M$ be a spacelike umbilic point and suppose the discriminant 
of the equation of the lines of principal curvature has an $A_{2k+1}^+$-singularity at ${\rm p}$. Then the multiplicity of ${\rm p}$ is $k$.

{\rm (2)} Let ${\rm p}\in M$ be a timelike umbilic point and suppose that the $LPL$ has a simple singularity at ${\rm p}$. Then 
the multiplicity of ${\rm p}$ is as follows: 

\begin{center}
\begin{tabular}{cc}
 Singularity of the $LPL$ at ${\rm p}$ & $m_u({\rm p})$\\
 $A^-_{2k-1}$&$k$\\
 $D_{2 k + 1}$ or $D^+_{2k}$ &$2$\\
 $D^-_{2k}$&$2$ or $k$\\
 $E_7$& $3$
\end{tabular}
\end{center}
\end{theo}

\begin{proof}
(1) In the Riemannian region we have $E \neq 0$, so
\begin{equation}\label{eq:mc}
	m_u({\rm p}) = m(\mathfrak{b},\mathfrak{c}) = \frac{1}{2} m(\mathfrak{b}^2,\mathfrak{c}) = \frac{1}{2} m(\mathfrak{b}^2-4 \mathfrak{a} \mathfrak{c},\mathfrak{c}).
\end{equation}

Similarly, as $G \neq 0$, it follows that $m_u({\rm p}) = m(\mathfrak{b},\mathfrak{a})$. We have
\begin{equation}\label{eq:mb}
	m_u({\rm p}) = \frac{1}{2} \left(m(\mathfrak{b},\mathfrak{c})+m(\mathfrak{b},\mathfrak{a})\right) = \frac{1}{2} m(\mathfrak{b},\mathfrak{a} \mathfrak{c}) = \frac{1}{2} m(\mathfrak{b},\mathfrak{b}^2-4\mathfrak{a} \mathfrak{c}).
\end{equation}

As the discriminant has an $A^+_{2k-1}$-singularity at $\rm p$, there is a local diffeomorphism $h$ such that 
$
(\mathfrak{b}^2-4\mathfrak{a} \mathfrak{c})(h(u,v))= \pm (u^2+v^{2k}).
$ Also, if $j^1(\mathfrak{b} \circ h) = au+bv$ and $j^1(\mathfrak{c} \circ h) = cu+dv$, then $a \neq 0$ or $c \neq 0$. 
Suppose that $a \neq 0$ (the case $c\ne 0$ follows similarly using (\ref{eq:mc})). Then it follows from (\ref{eq:mb}) that
$$
m_u({\rm p}) = \frac{1}{2} m(\mathfrak{b},\mathfrak{b}^2-4\mathfrak{a} \mathfrak{c}) = \frac{1}{2} m(au+bv+O(2),u^2+v^{2k}) = k.
$$

(2) We take a local parametrisation of $M$ with $\rm p$ as in the proof of \mbox{Theorem \ref{theo:sing_umb_time}}
so the $LPL$ is the zero set of $\delta = l n$. 
As $F \neq 0$, it follows from Proposition \ref{prop:mult_formula} that
$m_u({\rm p}) = m (\mathfrak{a},\mathfrak{c}) = m(- n F, l F) = m(n, l).$
	
Suppose that the $LPL$ has an $A^-_{2k-1}$singularity at $\rm p$, so 
there is a local diffeomorphism $h$ such that
$\delta (h(u,v))=l(h(u,v)) n(h(u,v)) = (u+v^k) (u-v^k).$
Thus, $ m_u({\rm p}) =  m (u-v^k, u+v^k) = k. $

If the $LPL$ has a $D_{2 k + 1}$-singularity at $p$, then there is a local diffeomorphism $h$ such that
	$\delta (h(u,v))=l(h(u,v)) n(h(u,v))= v (u^2 \pm v^{2 k - 1}).$
It follows that $ m_u({\rm p}) =  m(v, u^2 \pm v^{2k-1}) = 2. $

Similarly, if the $LPL$ has a $D^+_{2 k}$, we get $ m_u({\rm p}) =  (v, u^2+ v^{2k-2}) = 2. $
When its singularity is of type $D_{2k}^-$, we have two possibilities: 
$ m_u({\rm p}) =m(v,u^2-v^{2k-2})=2$ or 
$ m_u({\rm p}) =m(v(u\pm v^{k-1}), u\mp v^{k-1})=k.$

At an $E_7$-singularity we have $ m_u({\rm p}) = m (u, u^2+v^3) = 3.$
\end{proof}

A consequence of Theorems \ref{theo:LDMilnor} and \ref{theo:LPLSimple} is the following. 

\begin{theo}
Given an integer $k\ge 1$, there are spacelike,  timelike and 
lightlike umbilic points 
on surfaces in $\mathbb R^3_1$ 
that have multiplicity $k$.
\end{theo}

\begin{proof} 
The example $z=f(x,y)= x^3-x y^{k+1}$ (resp. $y=f(x,z)=x^3+x z^{k+1}$) 
in  the proof of Theorem \ref{theo:sing_umb_time} 
gives a surface with a spacelike (resp. timelike) umbilic point that has multiplicity $k$.

Consider a surface with an $LD$ that has an $A_k$-singularity at a lightlike umbilic point ${\rm p}$. 
	The Milnor number of the $LD$ at ${\rm p}$ is $k$, and by Theorem \ref{theo:LDMilnor}, 
	$m_u({\rm p})=k$. An example of such a surface is the graph of  
	$
	z=f(x,y)=\frac{1}{\sqrt{2}}(x+y)+\frac{1}{3}x^3+\frac{1}{k+2}y^{k+2}
	$ 
	with a lightlike umbilic point at the origin. The $LD$ is given by 
	$$
	\delta(x,y)=f_x(x,y)^2+f_y(x,y)^2-1=
	\frac{2}{\sqrt{2}}x^2+x^4+\frac{2}{\sqrt{2}}y^{n+1}+y^{2k+2}
	$$
	and has an $A_k$-singularity at the origin.
\end{proof}
\section{Applications}\label{sec:appl}

Bifurcations of codimension one umbilic points on surfaces in $\mathbb R^3$ 
are studied in \cite{GarciaSotoGut} (the results hold for spacelike umbilics on surfaces in $\mathbb R^3_1$).
We follow in this section as closely as possible the notations in \cite{GarciaSoto,GarciaSotoGut} and take the surface as the graph of a function $z=f(x,y)$ with 
$$
\begin{array}{rl}
f(x,y)=&\frac{\kappa}{2}(x^2+y^2) +
\frac{a}6x^3+\frac{b}{2}xy^2	+\frac{c}{6}y^3
+\frac{d_{40}}{24}x^4+\frac{d_{31}}{6}x^3y
+\frac{d_{22}}{4}x^2y^2+\frac{d_{13}}{6}xy^3+\frac{d_{04}}{24}y^4\\
&+\frac{d_{50}}{120}x^5+\frac{d_{41}}{24}x^4y
+\frac{d_{32}}{12}x^3y^2+\frac{d_{23}}{12}x^2y^3+\frac{d_{24}}{24}xy^4
+\frac{d_{05}}{120}y^5+O(6).
\end{array}
$$

From \cite{GarciaSotoGut}, umbilic points of codimension one are 
$$
\begin{array}{rcl}
{\rm D}^1_2&:& cb(b-a)\neq 0\, \mbox{ and either }\delta_1=(\frac{c}{2b})^2-\frac{a}{b}+1 = 0 
\mbox { or } a=2b;\\
{\rm D}^1_{2,3}&:&  a=b \neq 0 \mbox { and } \chi =c d_{31}-(d_{22}-d_{40}+2 \kappa^3)b \neq 0.
\end{array}
$$

In generic 1-parameter families of surfaces, at a ${\rm D}^1_2$-umbilic there is a change from a lemon to a monstar umbilic (so there should be only one umbilic point concentrated at a ${\rm D}^1_2$-umbilic). 
At a ${\rm D}^1_{2,3}$-umbilic there is a birth (and death) of two umbilics, one is a star and the other a monstar (so there should be two umbilic points concentrated at a ${\rm D}^1_{2,3}$-umbilic).

\begin{prop}
We have $m_u({\rm D}^1_2)=1$ and $m_u({\rm D}^1_{2,3})=2$.
\end{prop}

\begin{proof}
We have $E\ne 0$ so 
$m_u({\rm p}) = m (\mathfrak{b},\mathfrak{c})$. 
From \cite{GarciaSotoGut}, the 1-jet of the BDE of the lines of principal curvature at a $D^1_2$-umbilic is $j^1\omega_P=(-bv,(b-a)u+cv,bv)$. Therefore,   $m_u({\rm D}^1_2)=1$ as $b(b-a)\ne 0$.

At a  ${\rm D}^1_{2,3}$-umbilic,  using the expression of $j^2\omega_P$ in \cite{GarciaSotoGut} and the conditions for the 
${\rm D}^1_{2,3}$-umbilic, we find that $m_u({\rm D}^1_{2,3})=2$.
\end{proof}

The  codimension two umbilic points are studied in  \cite{GarciaSoto} and are as follows:

$$
\begin{array}{rcl}
{\rm D}^2_{1}&:& c=0 \mbox{ and } a=2b \neq 0;\\
{\rm D}^2_{2p}&:& a=b\ne 0,\,  \chi = 0, \xi \neq 0   \mbox{ and } \xi b<0;\\
{\rm D}^2_{3}&:& a=b\neq 0,\, \chi = 0, \xi \neq 0 \mbox{ and } \xi b > 0;\\
{\rm D}^2_{h}&:& (a=b=0  \mbox{ and } cd_{31} \neq 0 )  \mbox{ or } 
(b=c=0  \mbox{ and } ad_{13} \neq 0);
\end{array}
$$
with $\xi = 12\kappa^2 b^3 + (d_{32} - d_{50}) b^2 + (3d_{31}^2 - 3d_{31}d_{13} - c d_{41})b+ 3cd_{31}(d_{22} - \kappa^3)$.

\begin{prop}\label{prop:codim2Umb}
We have 
$m_u({\rm D}^2_{1})=1$, 
$m_u({\rm D}^2_{2p})=
	m_u({\rm D}^2_{3})=3$ and 
$m_u({\rm D}^2_{1})=2$.
\end{prop}

\begin{proof}
The proof follows by direct calculations using the jets of the coefficients $\mathfrak{b}$ and $\mathfrak{c}$ 
of the lines of principal curvature in \cite{GarciaSoto} and computing $m (\mathfrak{b},\mathfrak{c})$ taking into consideration the above conditions.
\end{proof}

\begin{rem}
{\rm 
1. It is shown in \cite{GarciaSoto} that the lines of principal curvature are topologically equivalent to a lemon, star and monstar at, respectively, the umbilics  
${\rm D}^2_{1}$, ${\rm D}^2_{2p}$ and ${\rm D}^2_{3}$. 
This shows that the multiplicity of umbilics is not preserved by topological equivalence of the BDEs of the lines of principal curvature.

2. The generic bifurcations of the codimension 2 umbilics in \cite{GarciaSoto} are yet to be established. Proposition \ref{prop:codim2Umb} gives the 
maximum number of stable umbilics one expect to appear in their bifurcations.
}
\end{rem}

Our results also apply to singular surfaces parametrised by corank 1 map-germs at the singular points. 
We call the direction of the 1-dimensional tangent space of the surface at the singular point $\rm p$ 
the tangential direction. If the tangential direction is spacelike or timelike,  
we can choose a system of coordinates so that $E\ne0 $ at $\rm p$ and compute $m_u({\rm p})$ 
using Proposition \ref{prop:mult_formula}. Further work is needed 
to extend our results to the case when the tangential direction is lightlike 
and to corank 2 singular surfaces as $E=F=G=0$ in both cases at the point of interest.

Consider the case of a surface in $\mathbb R^3_1$ with a cross-cap singularity (\cite{DiasTari}). Then the BDE of the lines of principal 
curvature extends to the singular point by considering equation (\ref{eq:principalLD}). 

\begin{prop}\label{prop:cross-capMul}
Let $M\subset \mathbb R^3_1$ be a surface with a cross-cap singularity at ${\rm p}$ and suppose that the tanegential 
direction at ${\rm p}$ is spacelike or timelike. Then $m_u({\rm p})=1$.
\end{prop}

\begin{proof}
One can choose a parametrisation 
of the cross-cap so that 
the 1-jet of the coefficients of the BDE of the lines of principal curvature is given by $j^1\omega_P=(0,-1/2x,y)$; see \cite{DiasTari}. Consequently, 
$m_u({\rm p})=1$. This means that there is only one umbilic point at the cross-cap. 
\end{proof}

Proposition \ref{prop:cross-capMul} is also valid for surfaces 
in $\mathbb R^3$ with a cross-cap singularity. Their lines of principal curvature are studied in \cite{GarciaSotoGut_CrossCap,Tari_Crosscap}.

\section{Appendix}\label{sec:appendix}

All the coefficients of the BDE of the lines of principal curvature vanish at an umbilic point. 
This type of BDEs is studied as follows. 
To a BDE (\ref{eq:genBDE})
is associated a surface 
$$
\mathcal M=\{(u,v,[p:q]\in U\times \mathbb RP^1: a(u,v)p^2+b(u,v)pq+c(u,v)q^2=0\}.
$$

The surface $\mathcal M$ is regular along the exceptional fibre $(0,0)\times \mathbb RP^1$ if, and only 
if, the discriminant function $\delta=b^2-4ac$ has a Morse $A_1^{\pm}$-singularity  at the origin (see \cite{bdes}).
Suppose this is the case. 
The bi-valued direction field determined by the BDE in $U$ lifts to a single field $\xi$ on $\mathcal M$. 
We write $F(u,v,p)=a(u,v)p^2+b(u,v)p+c(u,v)$ in the chart $q=1$ (the chart $p=1$ also needs to be considered). 
Then one can take
$$
\xi =F_p\frac{\partial}{\partial u}+pF_p \frac{\partial}{\partial v}-
(F_u+pF_v)\frac{\partial}{\partial p}.
$$ 

The exceptional fibre $(0,0)\times \mathbb RP^1$ is an integral curve of $\xi$ and the 
singularities of $\xi$ on this fibre are the roots of the cubic polynomial
$$
\phi(p)=F_u(0,0,p)+pF_v(0,0,p).
$$

Suppose that $\phi$ has simple roots. Then, at a given root ${p}$ of $\phi$, the singularity 
of $\xi$ is a saddle (resp. node) if $\phi'(p)\alpha(p)>0$ (resp. $<0$), where
$$
\alpha(p)=
a_v(0,0)p^2+ (\frac{1}{2}b_v(0,0)+a_u(0,0))p+
\frac{1}{2}b_u(0,0).
$$

Suppose that the discriminant of the BDE has a Morse $A_1^+$-singularity and that 
the roots of $\phi$ are simple.
Then, when $\phi$ has 3 roots, they are either all of type saddle and the umbilic is called {\it star}, or two are saddles and one is a node and the umbilic is called {\it monstar}. When $\phi$ has one root, it is a saddle and the umbilic is called {\it lemon}. 

When $\delta$ has an $A_1^-$-singularity and the roots of $\xi$ are simple, they are all of type saddle or node and we get all the 
possible combinations in Figure \ref{fig:TimelikeUmb}.
The Morse singularity type $A_1^+$ or $A_1^-$ together with the number and type of the singularities of $\xi$ determine the topological configuration of the BDE.
Furthermore, the configuration is completely determined by the 1-jets of the coefficients of the BDE at the umbilic point 
(\cite{bdes}), see Figure \ref{fig:Darboux} and Figure \ref{fig:TimelikeUmb}.  
\begin{figure}[h!]
	\begin{center}
		\includegraphics[scale=0.6]{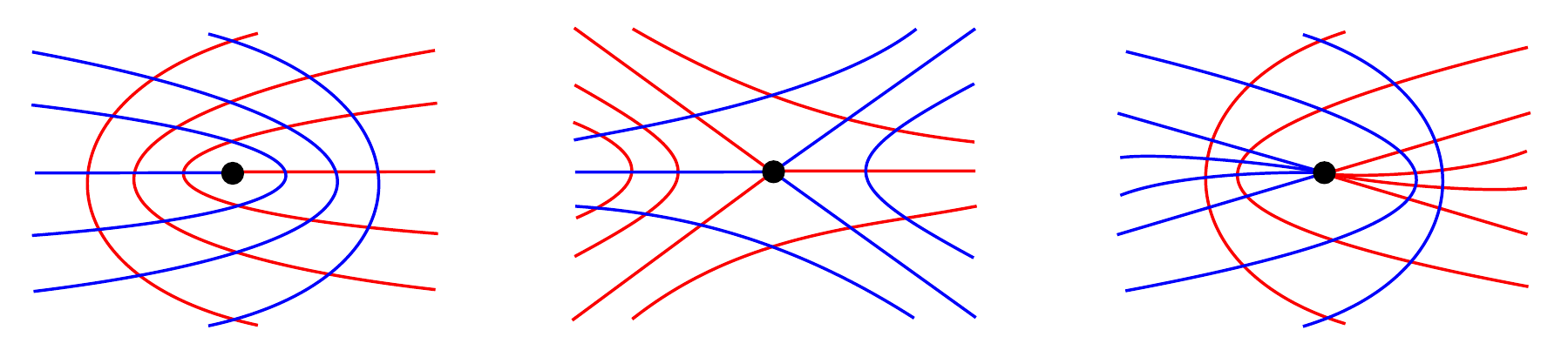}
		\caption{Generic configurations of the lines of principal curvature at a spacelike umbilic point: 
			lemon (first figure), star (second figure), monstar (third figure).}
		\label{fig:Darboux}
	\end{center}
\end{figure}
\begin{figure}[h!]
	\begin{center}
		\includegraphics[width=14cm,height=2.5cm]{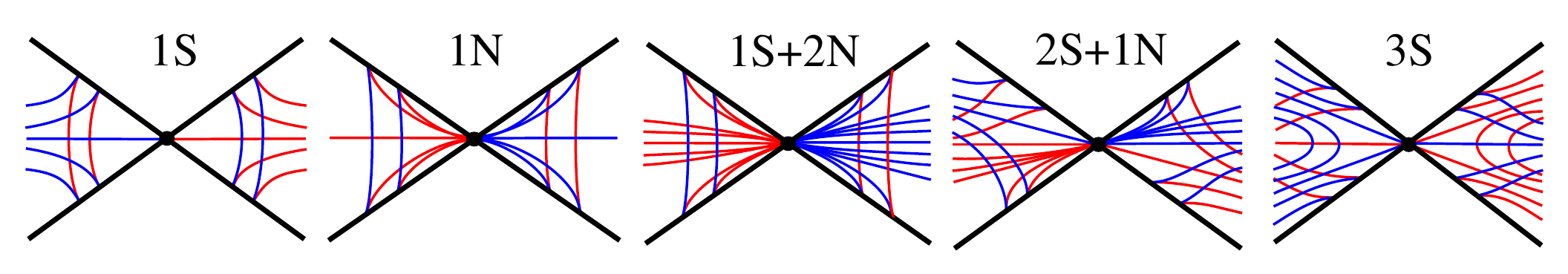}
		\caption{Generic configurations of the lines of principal curvature at a timelike umbilic point (S for a saddle and N for a node singularity of $\xi$).}
		\label{fig:TimelikeUmb}
	\end{center}
\end{figure}

When one of the coefficients of the BDE (\ref{eq:genBDE}) is not zero at the origin, a similar approach to the above is used to study such BDEs; see \cite{davbook}. The topologically stable phenomena are those where $M$ is a regular surface and the vector field $\xi$ is regular or 
has a saddle, node or focus singularity (we also require that the separatrices of the saddle or the node are not tangent to the 
kernel of the projection to the plane). The singularities of the BDE corresponding to the saddle/node/focus of $\xi$ are called {\it well-folded singularities}.


\subsection{Partition of the space of cubic forms}\label{ssec:config}

The analysis at a spacelike umbilic point  is the same as that on surfaces in the Euclidean space.
We can choose a suitable coordinate system so that 
$M$ is the graph of a function $z=f(x,y)$, so it is 
parametrised locally by $(x,y)\to {\bx}(x,y)$, with 
\begin{eqnarray}\label{parm:spacelikeUmb}
	{\bx}(x,y)=(x,y,\frac{\kappa}{2}(x^2+y^2)+C(x,y)+O_4(x,y)), 
\end{eqnarray}
where $(x,y)$ is in some neighbourhood $U$ of the origin in $\mathbb R^2$, $C$ is a homogeneous cubic.
The cubic form
$C(x,y)$ can be taken as  the real part of $z^3+\beta z^2\overline{z}$, with $z=x+iy$ and $\beta=s+it$ (see for example \cite{bruce84,porteous}). 
Then, 
\begin{equation}\label{eq:Cubicspace}
	C(x,y)=(1+s)x^3-tx^2y+(s-3)xy^2-ty^3.
\end{equation}

The following result is known (see for example \cite{bruce84,bdes,porteous,sotogut}). We include it for completion. The discriminant $\delta$ is that of BDE (\ref{eq:principalBDE}) and the cubic $\phi$ and the quadratic $\alpha$ are as above.

\begin{prop}\label{prop:spacelikeumbPart}
	Suppose that ${\rm p}$ is a spacelike umbilic point and that $M$ is parametrised locally as 
in	$(\ref{parm:spacelikeUmb})$ at ${\rm p}$. Then the $\beta$-plane is stratified by the following curves:
	
-- The discriminant has a singularity more degenerate than $A_1^{+}$: 
	$s^2+t^2=9$.
	
-- $\phi$ has a repeated root: $\beta(\theta)=-3(2e^{2i\theta}+e^{-4i\theta}), \theta\in \mathbb R$ {\rm(}the outer hypocycloid in {\rm Figure  \ref{fig:PartitionA_1^+})}. 
	
-- $\alpha$ and $\phi$ have a common root: $s^2+t^2=9$.
\end{prop}
\begin{figure}[h!]
	\begin{center}
		\includegraphics[scale=0.4]{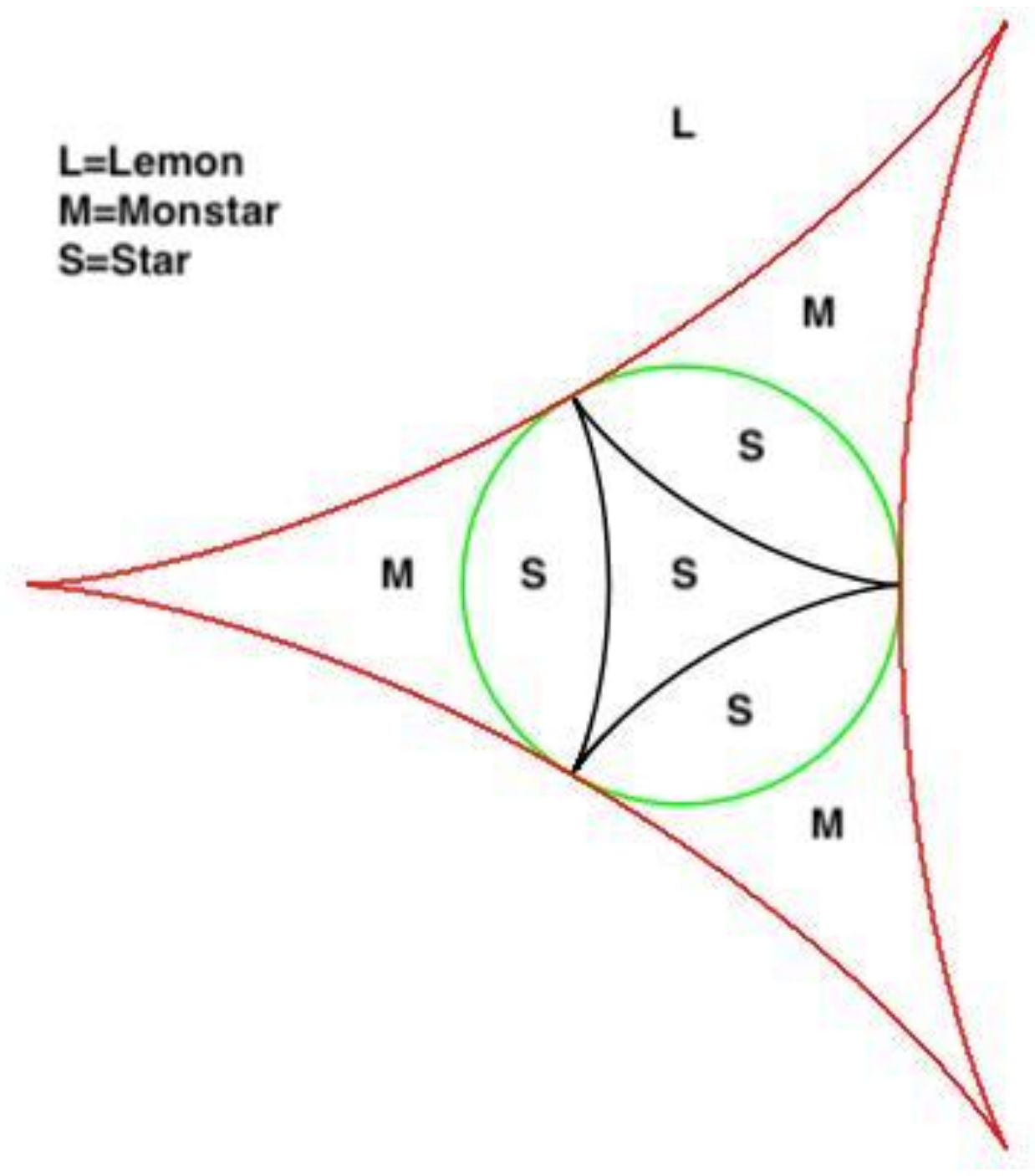}
		\caption{Partition of the $\beta$-plane by the exceptional curves.}
		\label{fig:PartitionA_1^+}
	\end{center}
\end{figure}

\begin{proof}The 1-jet of the BDE of the lines of principal curvature is given by 
	\begin{equation}\label{eq:1jetBDEspace}
		j^1\omega=(tx+(3-s)y,-2((s+3)x+ty),-(tx+(3-s)y)).
	\end{equation}

	Then 
	$\phi(t)=(3-s)p^3-tp^2-(9+s)p-t$ and
	$\alpha(t)=(3-s)p^2-3-s$ and the result follows by straightforward calculations.
	The configurations of the lines of principal curvature are stable and are all topologically equivalent 
	for $C$ in a given open stratum in the $\beta$-plane 
	(Figure \ref{fig:Darboux} and \mbox{Figure \ref{fig:PartitionA_1^+}}).
\end{proof}

At a timelike umbilic point on a generic surface, the $LPL$ is a curve which has a Morse singularity 
of type $A_1^-$ (a node). 
There are two principal directions at each point on one side of the $LPL$ and none on the other.
The generic topological models of the configurations of the lines of principal curvature at a timelike umbilic point are as in Figure \ref{fig:TimelikeUmb}. We give here a characterisation of a timelike umbilic in terms of the cubic form in a parametrisation of the surface, in an analogous way to the spacelike case. 
We take the surface locally as the graph a function $y=f(x,z)$, so it can be parametrised by $(x,z)\mapsto {\bx}(x,z)$ with 
\begin{eqnarray}\label{parm:timelikeUmb}
	{\bx}(x,z)=(x,\frac{\kappa}{2}(x^2-z^2)+C(x,z)+O_4(x,z),z).
\end{eqnarray}

A linear factor $ax+bz$ of the cubic form $C$ is said to be a spacelike (resp. timelike or lightlike)  root  if the tangent direction to the curve
$\gamma = \{ \bx (x,z) : ax+bz = 0\}$
at the origin is spacelike (resp. timelike or lightlike). Therefore, a linear factor $ax+bz$ of $C$ is a spacelike  (resp. timelike or lightlike)  root if, and only if, $b^2 > a^2$ (resp. $b^2<a^2$ or $b^2=a^2$).

We have the following result. (Observe that Lorentizian changes of coordinates and dilatations preserve the lines of principal curvature of a given surface.)

\begin{prop}\label{prop:cubic}
Suppose that the cubic form $C(x,z)$ in \mbox{\rm (\ref{parm:timelikeUmb})} is not identically zero. Then it can be reduced by a Lorentzian change of coordinates in $\mathbb R^3_1$ and a dilatation to one
	of the following forms
	$$\begin{array}{rl}
		{\rm (i)}& x(x^2 +s xz+t z^2),\\
		{\rm (ii)}&  z(t x^2 +s xz+z^2),\\
		{\rm (iii)}& (x\pm z)(x^2 +s xz+tz^2),\\
		{\rm (iv)}& x^2 z,\\
	\end{array}
	$$
	with $s,t \in \mathbb R$.
\end{prop}

\begin{proof} We write $C(x,z)=a_1x^3+a_2x^2z+a_3xz^2+a_4z^3$ and make the change of coordinates 
	$(X,Y,Z)\mapsto (x,y,z)$ in $\mathbb R^3_1$,  with 
	
	$$
	\left(
	\begin{array}{c}
		x\\
		y\\
		z
	\end{array}
	\right)=
	\left(
	\begin{array}{ccc}
		\cosh\theta&0&-\sinh\theta\\
		0&1&0\\
	-	\sinh\theta&0&\cosh\theta
	\end{array}
	\right)
	\left(
	\begin{array}{c}
		X\\
		Y\\
		Z
	\end{array}
	\right).
	$$
	
	Then
	$C(X,Z)=A_1X^3+A_2X^2Z+A_3XZ^2+A_4Z^3$, with
	{\small
		$$
		\begin{array}{l}
			A_1=a_1\cosh^3\theta-a_2\cosh^2\theta\sinh\theta+a_3\cosh\theta\sinh^2\theta-a_4\sinh^3\theta,\\
			A_2=a_2\cosh^3\theta-(3a_1+2a_3)\cosh^2\theta\sinh\theta+(2a_2+3a_4)\cosh\theta\sinh^2\theta-a_3\sinh^3\theta,\\
			A_3=a_3\cosh^3\theta-(2a_2+3a_4)\cosh^2\theta\sinh\theta+(3a_1+2a_3)\cosh\theta\sinh^2\theta-a_2\sinh^3\theta,\\
			A_4=a_4\cosh^3\theta-a_3\cosh^2\theta\sinh\theta+a_2\cosh\theta\sinh^2\theta-a_1\sinh^3\theta.
		\end{array}
		$$
	}

	If $C$ has a timelike root, we can 
	choose $\theta$ so that the root is tangent to the curve $\{\bx (X,Z) : X = 0\}$, i.e., we can set $A_4=0$.
	We can then rescale and obtain the reduced forms (i), (iv) or
$XZ (dX+Z)$. For the later, by another choice of $\theta$, we can rewrite the cubic as in (i), (ii) or (iii) when $d^2 > 1$, $d^2 < 1$ or $d^2 = 1$, respectively.
	
	If $C$ has a spacelike root and no timelike roots, we choose $\theta$ so that the root
	is tangent to the curve $\{\bx (X,Z) : Z = 0\}$, i.e.,
 we can set $A_1=0$. We can then rescale and obtain the reduced form (ii) as  $A_4 \neq 0$.
	
	If one root of $C$ is lightlike, then $u+v$ or $u-v$ is a factor of $C$. We then 
	rescale to get the reduced form (iii).

\end{proof}

We identify a BDE $\omega$ with its coefficients and write $\omega=(a,b,c)$.

\begin{prop}\label{prop:1jet_bde_timelike_umb}
	The 1-jet of the BDE of the lines of principal curvature at a timelike umbilic point with the cubic $C$ equivalent to 
	one of the reduced forms in {\rm Proposition \ref{prop:cubic}(i)-(iv)} are equivalent, respectively,  to
	$$
	\begin{array}{cl}
		{\rm (i)}&(s x+t z,\, (3+t )x+s z,\, s x+t z),\\
		{\rm (ii)} &(t x+s z,\, s x+(3+t )z,\, t x+s z),\\
		{\rm (iii)}&((s \pm 1)x+(t \pm s )z,\, (3\pm s +t )x+(s \pm 3t \pm 1)z,\, (s \pm 1)x+(t \pm s )z),\\
		{\rm (iv)}&(x,z,x).
	\end{array}
	$$
	The 1-jets {\rm (i)} and {\rm (ii)} are equivalent by the change of coordinates $(x,z)\mapsto (z,x)$. 
	
	\smallskip
	The exceptional curves in the $(s,t)$-plane for the case {\rm (i) (}and {\rm (ii))} are:
	
-- The discriminant has a singularity more degenerate than $A_1^{-}$: $s^2-t^2-3t=0$.
	
-- $\phi$ has a repeated root:
	$8s^4-\frac{61}{4}s^2t^2+8t^4-39s^2t+36t^3-9s^2+54t^2+27t=0.$
	
-- $\alpha$ and $\phi$ have a common root: $(s+t+1)(s-t-1)=0$.
	
	The partition of the $(s,t)$-plane by these curves is as in {\rm Figure \ref{fig:Part_TimelikeUmb}}. In the open regions of this partition, the BDE of the lines of 
	principal curvature is topologically determined by its 1-jet and has one of the configurations in {\rm Figure \ref{fig:TimelikeUmb}}.
	
	\smallskip
Case {\rm (iii)} does not occur on generic surfaces as $\alpha$ and $\phi$ have a common root. They have one common root unless $(s,t)=(0,-1)$ in which case they have two roots in common. 
	The 
	exceptional curves in the $(s,t)$-plane in this case are:
	
	- The discriminant has a singularity more degenerate than $A_1^{-}$: $(t-1)(1+t \mp s)=0.$
	
	- $\phi$ has a repeated root: $3s^2-4t^2-4t-4=0.$
	
	\smallskip
	In case {\rm (iv)}, the discriminant has a singularity $A_1^-$ and the lines of principal curvature has the configuration {\rm (1S)} in {\rm Figure \ref{fig:TimelikeUmb}}.
\end{prop}

\begin{proof}
	The 2-jet of the discriminant of the BDE with a cubic as in Proposition \ref{prop:cubic}(i) is, up to a scalar multiple,
	$j^2\delta=((2s+t+3)x+(s+2t)z)((2s-t-3)x+(2t-s)z)$ and this has a singularity more degenerate than $A_1^{-}$ 
	if, and only if, $s^2-t^2-3t=0$ (blue hyperbola in Figure \ref{fig:Part_TimelikeUmb}, first and last figure).
	
	We  have 
	$\phi=t p^3+2s p^2+(2t +3)p+s $ and its discriminant is as stated in the proposition 
	(red curve in Figure \ref{fig:Part_TimelikeUmb}, second and last figure). 
	
	We also have $\alpha=2t p^2+3s p+t+3$ and $\phi$ and $\alpha$ have a common root  if, and only if, 
	$(s+t+1)(s-t-1)=0$ (green  curve in Figure \ref{fig:Part_TimelikeUmb}, first and last figure), 
	or $s^2-t^2-3t=0$, which is where $\delta$ has a degenerate singularity.
	
	\begin{figure}[h]
		\centering
		\includegraphics[scale=0.2]{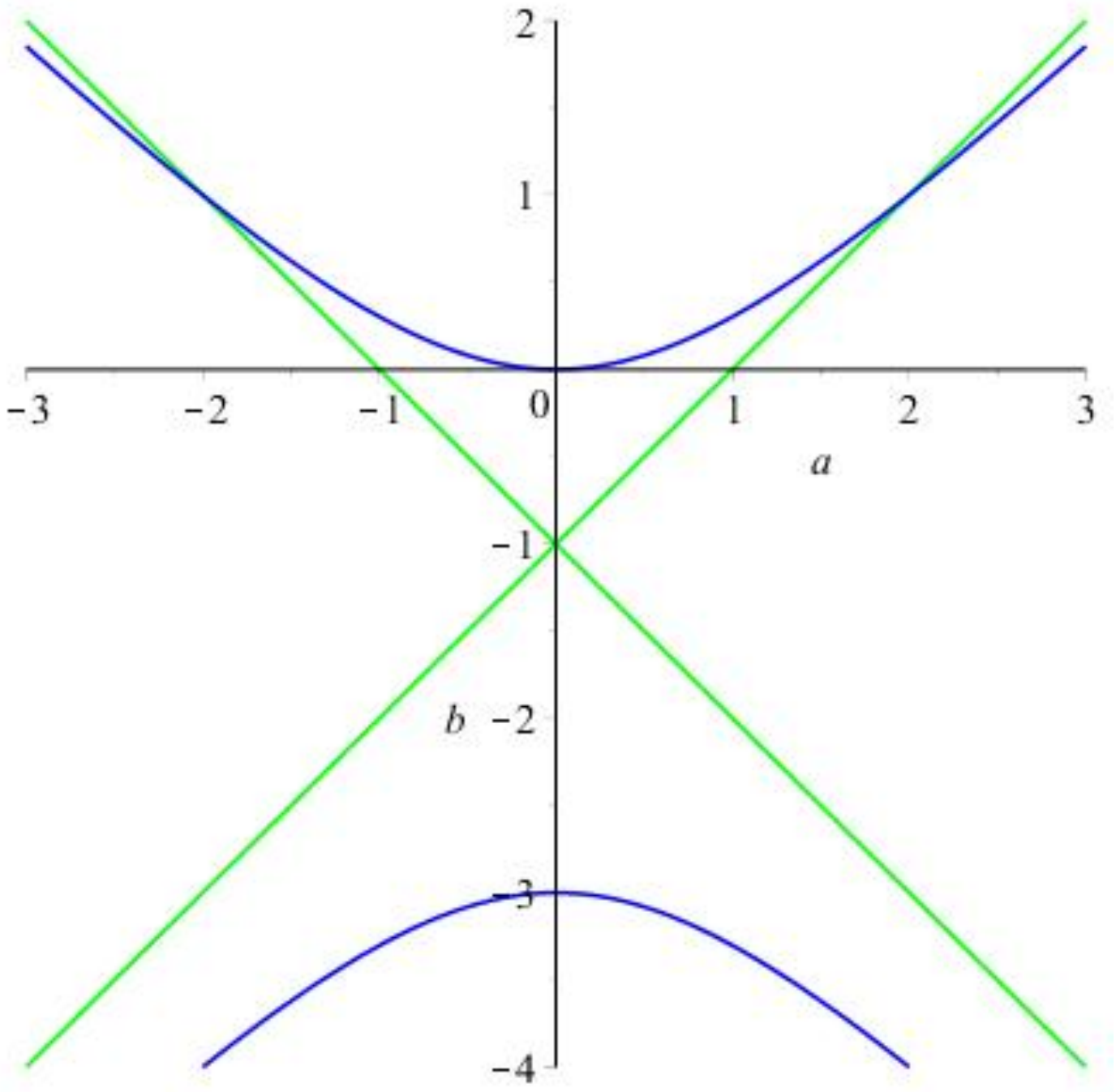}\quad
		\includegraphics[scale=0.2]{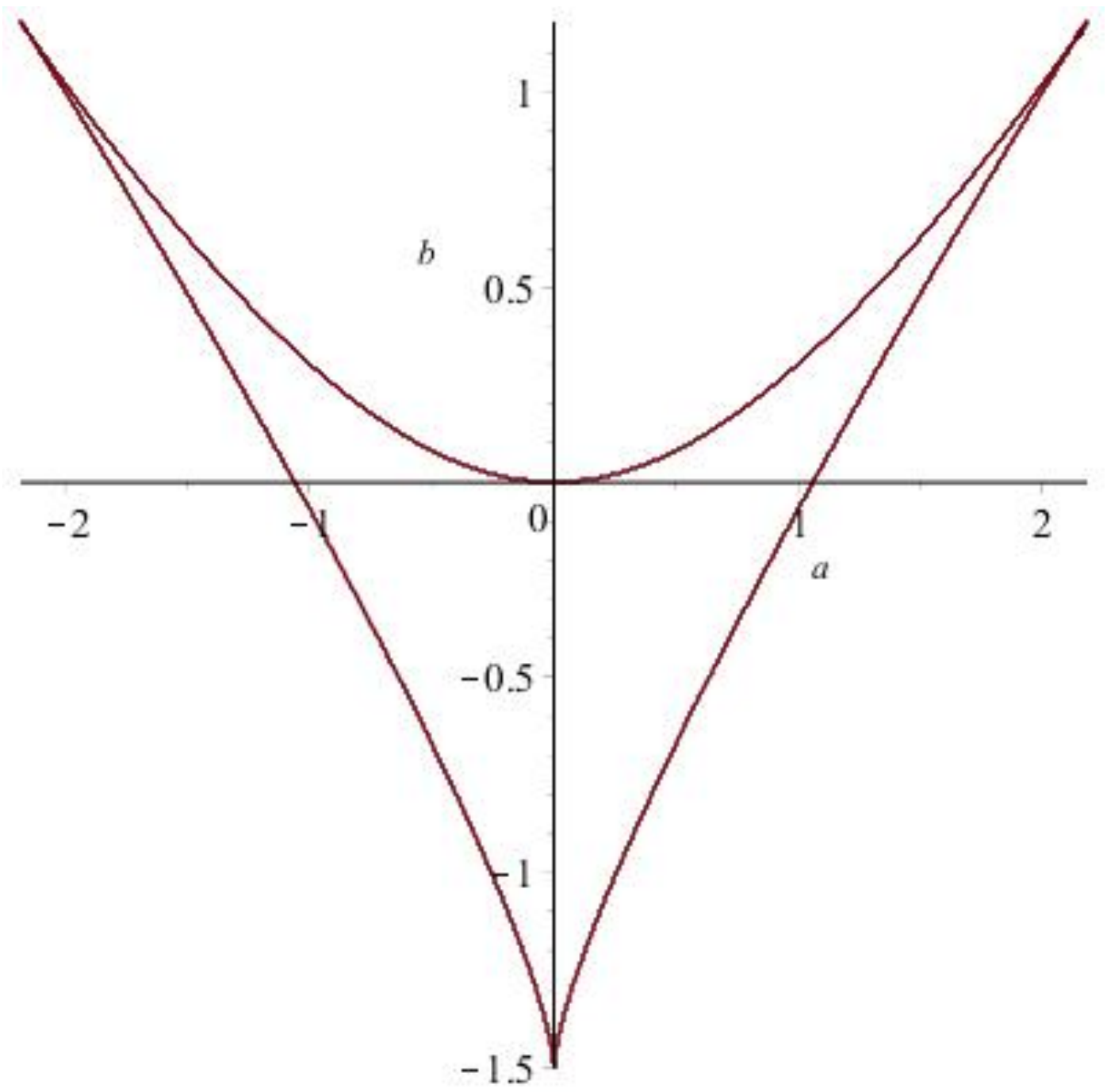}\quad 
		\includegraphics[scale=0.17]{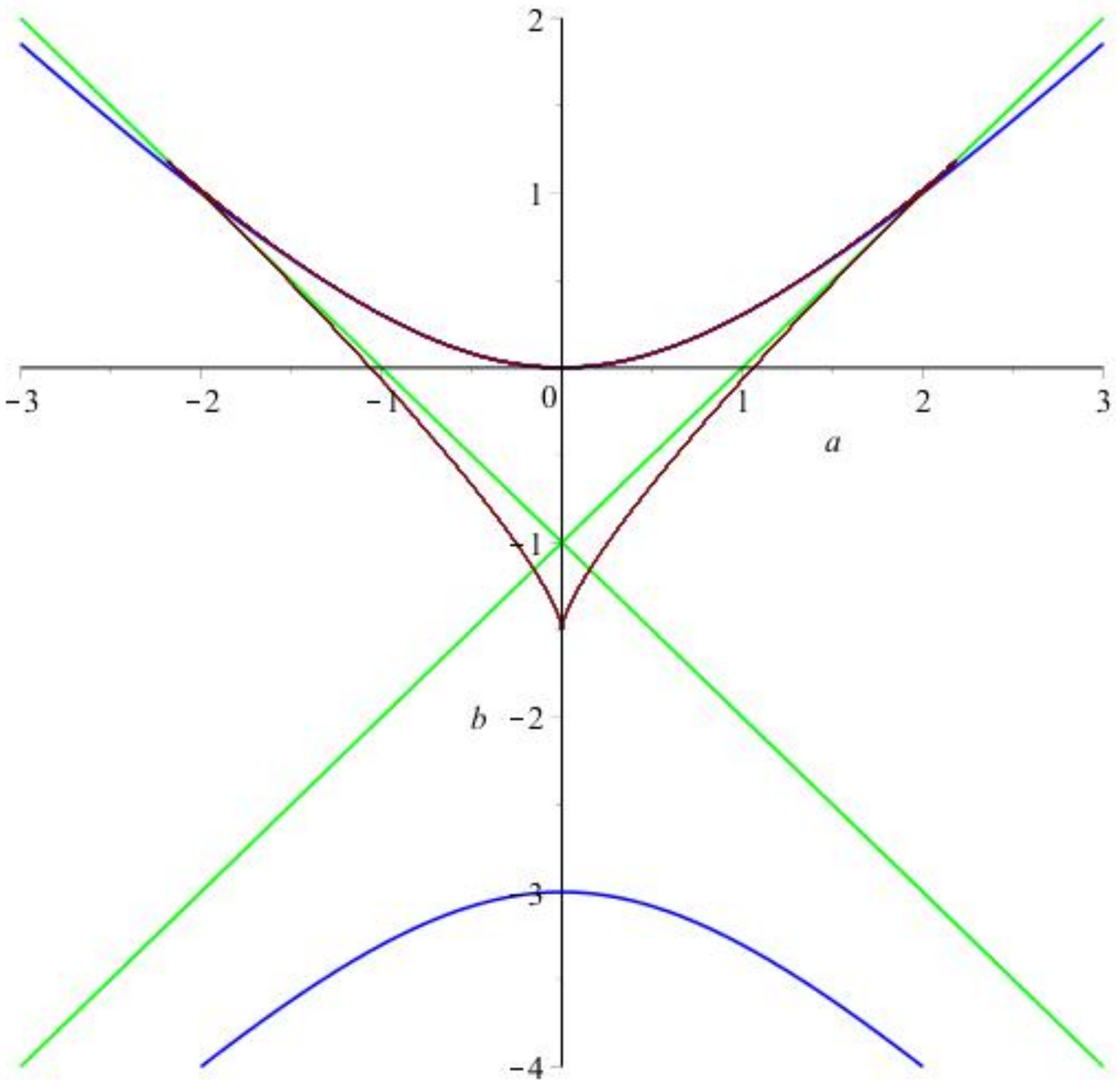}
		\caption{Partition of the $(s,t)$-plane by the exceptional curves, case (i) in \mbox{Proposition \ref{prop:1jet_bde_timelike_umb}}. The figures are drawn using Maple.}
		\label{fig:Part_TimelikeUmb}
	\end{figure}
	
	We determine in each open region in Figure \ref{fig:Part_TimelikeUmb} (last figure) the number and the type of the singularities of the lifted vector field $\xi$. Figure \ref{fig:Part_TimelikeUmb_Lines} shows the configurations of the lines of principal curvature in each open region of \mbox{Figure \ref{fig:Part_TimelikeUmb}}.
	\begin{figure}[h]
		\centering
		\includegraphics[scale=0.5]{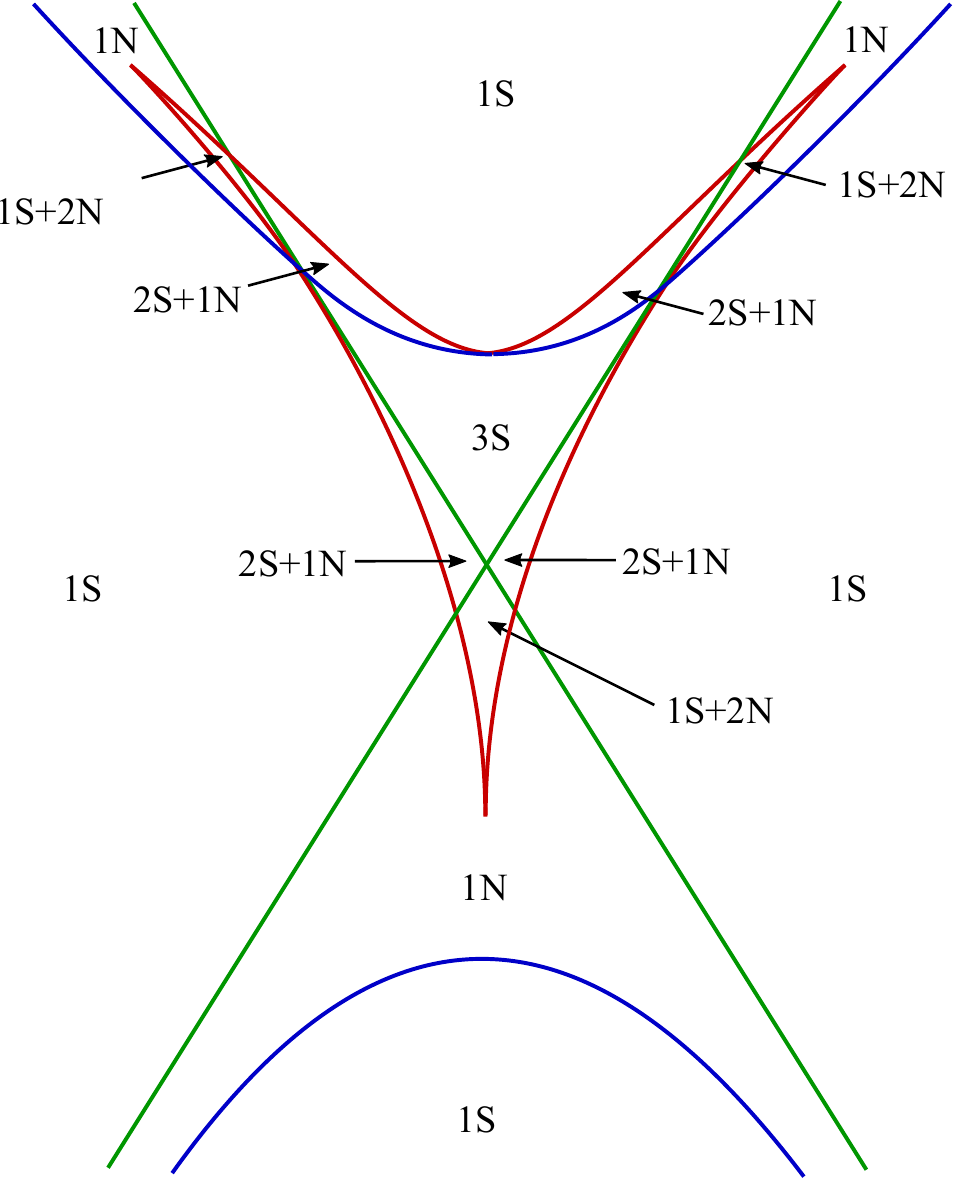}
		\caption{Topological types of the lines of principal curvature in the open regions in the $(s,t)$-plane.}
		\label{fig:Part_TimelikeUmb_Lines}
	\end{figure}
	
	For the case (iii) in Proposition \ref{prop:cubic}, the 2-jet of the discriminant 
	of the BDE is given by 
	$j^2\delta=(s\mp t\mp 1)(x\pm z)((3s\pm t\pm 5)x+(\mp 3s+5t+1)z)$. 
	Its has a singularity more degenerate than $A_1^{-}$ if, and only if, $s\mp t\mp 1=0$ 
	or $t=1$ (blue curve in \mbox{Figure \ref{fig:Part_TimelikeUmbLightDirc}} for the $(+)$ case).
	
	We have $\phi=(p\pm 1)(s p^2\pm t p^2\pm s p+2t p+s +2p\pm 1)$. It has a repeated root 
	when $3s^2-4t^2-4t-4=0$ (red hyperbola in Figure \ref{fig:Part_TimelikeUmbLightDirc}), or when 
	$s\mp t\mp 1=0$, that is, when the singularity of $\delta$ is more degenerate than $A_1^{-}$. 
	
	Here
	$\alpha=\frac{1}{2}(p\pm 1)(2s p\pm 2t p\pm s+t+3)$.
	Therefore, $\phi$ and $\alpha$ have always at least one common root. This means that the umbilic is 
	of codimension $\ge 1$. They have two roots in common when $(s,t)=(0,-1)$.
\end{proof}

\begin{figure}[h]
	\begin{center}
		\includegraphics[scale=0.25]{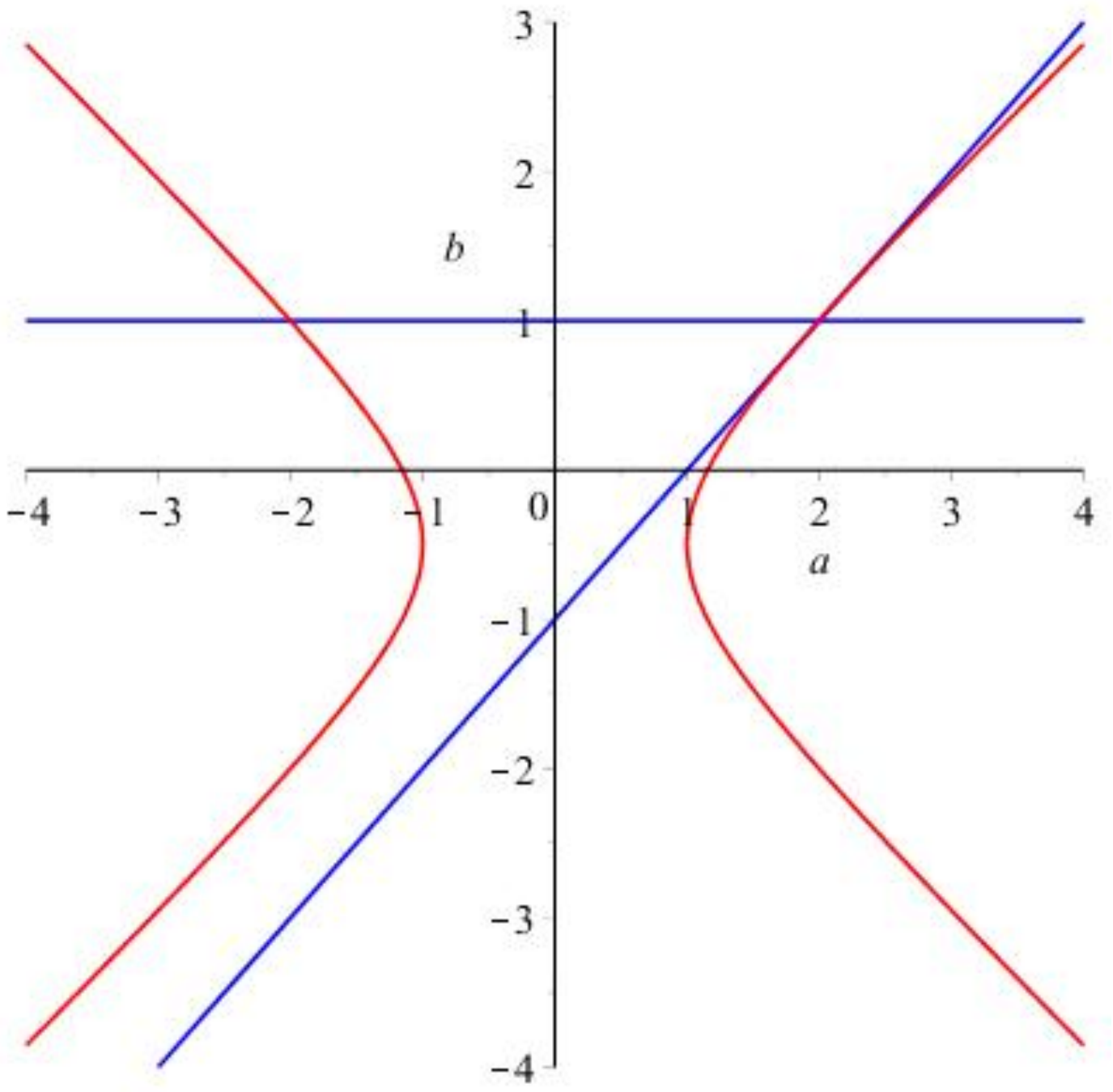}
		\caption{Partition of the $(s,t)$-plane by the exceptional curves for case (iii)(+) in \mbox{Proposition \ref{prop:1jet_bde_timelike_umb}}. For the $(-)$ case, 
			the blue slanted line is replaced by its reflection with respect to the $t$-axis.}
		\label{fig:Part_TimelikeUmbLightDirc}
	\end{center}
\end{figure}

\begin{rem}
	{\rm 
		The change of coordinates $(x,z)\mapsto (z,x)$ in Proposition \ref{prop:1jet_bde_timelike_umb} does not preserve the metric, however it preserves the configuration of the BDEs 
		as well as the invariant defined in this paper, so it is enough to work with the forms (i), (iii). 
		(The BDEs (i) and (iii) are not topologically equivalent.)
	}
\end{rem}

\begin{rem}
	{\rm The green and blue curves in Figures \ref{fig:PartitionA_1^+} and \ref{fig:Part_TimelikeUmb_Lines}, respectively, represent the parameter values for which the multiplicity is greater than 1. At such points one needs to use higher jets of $f$ 
		in order to compute the multiplicity of the umbilic point.}
\end{rem}

Lightlike umblic points are not stable, however we need the following notation and results.
When ${\rm p}$ is a lightlike umbilic point, we can parametrise $M$ locally at ${\rm p}$ by $(x,y)\mapsto {\bx}(x,y)$ with $(x,y)$ in some neighbourhood of the origin and
\begin{equation}\label{paramLightUmb}
	{\bx}(x,y)=(x,y, \pm x+a_{22}y^2+a_{30}x^3+a_{31}x^2y+a_{32}xy^2+a_{33}y^3+O_4(x,y));
\end{equation}
see Theorem 2.6.4 in \cite{MarcoThesis}. 
The following follows by straightforward calculations.

\begin{prop}\label{prop:Lightlikeumbilic}
	Let ${\rm p}$ be a lightlike umbilic point and suppose that $M$ is parametrised locally at ${\rm p}$ as in {\rm (\ref{paramLightUmb})}. 
	Then,
	
	{\rm (i)} The 1-jet of the BDE of the lines of principal curvature is given by 
	$$j^1\omega=2((2a_{22}^2-a_{32})y-a_{31}x)dy^2-2(3a_{30}u+a_{31}v)dxdy.
	$$
	
	{\rm (ii)} The $LPL$ is singular at ${\rm p}$ and has generically an $A_3$-singularity.
	
	{\rm (iii)} The $LD$  is singular at ${\rm p}$ and has an $A_1^\pm $-singularity if, and only if, 
	$$
	6a_{22}^2a_{30}\pm 3a_{30}a_{32}\mp a_{31}^2\ne 0.
	$$
\end{prop}

\begin{acknow}
	MACF was financed by the Coordenação de Aperfeiçoamento de Pessoal de Nível Superior – Brasil (CAPES) – Finance Code 001.
	FT was partially supported by the 
	FAPESP Thematic project grant 2019/07316-0  and the CNPq research grant 303772/2018-2.
\end{acknow}


\noindent
MACF and FT: Instituto de Ci\^encias Matem\'aticas e de Computa\c{c}\~ao - USP, Avenida Trabalhador s\~ao-carlense, 400 - Centro,
CEP: 13566-590 - S\~ao Carlos - SP, Brazil.\\
E-mail: faridtari@icmc.usp.br

\end{document}